\documentclass[11pt, a4paper, oneside]{amsart}

\usepackage{amsfonts, amssymb, txfonts}
\usepackage{commath}
\usepackage{empheq}
\usepackage{amsmath, braket, eucal, amsthm, mathdots}

\usepackage[utf8]{inputenc}
\DeclareFontFamily{U}{min}{}
\DeclareFontShape{U}{min}{m}{n}{<-> udmj30}{}

\usepackage{dsfont}
\usepackage{multicol}
\usepackage{comment}
\usepackage{bbm}

\usepackage{xcolor}
\definecolor{myurlcolor}{rgb}{0.1,0.1,0.8}
\definecolor{mylinkcolor}{rgb}{0.05,0.05,0.4}

\usepackage{hyperref}
\hypersetup{colorlinks,
	linkcolor=mylinkcolor,
	citecolor=mylinkcolor,
	urlcolor=myurlcolor}

\usepackage{cleveref}

\usepackage[cmtip,all]{xy} %%required for using xymatrix (sometimes useful)

\usepackage{graphicx} % Required for inserting images
\theoremstyle{plain}

\usepackage{tikz,tikz-cd}
\usetikzlibrary{decorations.pathreplacing}
\usetikzlibrary{matrix,arrows,decorations.pathmorphing,decorations.markings}

\newtheorem{thm}{Theorem}[subsection]

\newtheorem{prop}[thm]{Proposition}
\newtheorem{lem}[thm]{Lemma}
\newtheorem{cor}[thm]{Corollary}

\newtheorem*{thmNMet-Brown-cat}{\Cref{thm:NMet-Brown-cat}}

\theoremstyle{definition}
\newtheorem{defn}[thm]{Definition}
\theoremstyle{remark}
\newtheorem{rem}[thm]{Remark}

\newtheorem{eg}[thm]{Example}

\newtheorem{notation}[thm]{Notation}

\usepackage{capt-of}

\usepackage{enumitem}

\usepackage{mathrsfs}

\usepackage{url}

\newcommand{\kk}{\Bbbk}  %the underlying ring
\newcommand{\N}{\mathbb{N}}
\newcommand{\Z}{\mathbb{Z}}
\newcommand{\id}{\mathrm{id}}
\newcommand{\Cyl}{\mathrm{Cyl}}
\newcommand{\Cylr}{\mathrm{Cyl}_{r}}
\newcommand{\fc}{\mathrm{fC}_\kk} %the category of filtered complexes
\newcommand{\Tot}{\mathrm{Tot}}
\newcommand{\bgrmod}{\mathrm{bgMod}_\kk} 		% bigraded modules over underlying ring
\newcommand{\spse}{\mathrm{Spse}_\kk}   %the category of spectral sequences
\newcommand{\Ee}{\mathcal E}
\newcommand{\cof}{\mathrm{Cof}}
\newcommand{\Ho}{\mathrm{Ho}}
\newcommand\smallbullet{\raisebox{-0.5ex}{\scalebox{1.8}{$\cdot$}}}
\newcommand{\pt}{\{{\smallbullet}\}}
\newcommand{\terminal}{!}
\newcommand{\short}{\mathrm{short}}
\newcommand{\plusx}{\otimes_1}
\newcommand{\inhom}[2]{[{#1},{#2}]_{\sup}}
\newcommand{\Rbar}{\overline{\mathbb{R}}_{\geq 0}}

\newcommand{\tuple}{\underline}
\newcommand{\dis}{\texttt{d}} %%%%distance
\newcommand{\dinfty}{\dis_{\sup}}
\renewcommand{\mod}{/\!\!}
\newcommand{\cn}[1]{\mathbf{#1}}

\newcommand{\demph}[1]{\emph{#1}}
\newcommand{\DiGraph}{\cn{DiGraph}}
\newcommand{\colim}{\mathrm{colim}}
\newcommand{\hocolim}{\mathrm{hocolim}}
\newcommand{\Met}{\cn{Met}}
\newcommand{\NMet}{\mathbb{N}\cn{Met}}
\newcommand{\RC}{\mathrm{RC}}
\newcommand{\tRC}{\texttt{RC}}

\newcommand{\MH}{\mathrm{MH}}
\newcommand{\PH}{\mathrm{PH}}
\newcommand{\GLMY}{\mathsf{PH}}

\title[Homotopy theories via the MPSS]{Homotopy theories\\ via the magnitude-path spectral sequence}

\author{Muriel Livernet}
\address{ML: Universit\'e Paris Cit\'e, Sorbonne Universit\'e, CNRS, IMJ-PRG, F-75013 Paris, France}

\author{Emily Roff}
\address{ER: School of Mathematics and Maxwell Institute for Mathematical Sciences, University of Edinburgh, Scotland}

\author{Sarah Whitehouse}
\address{SW: School of Mathematical and Physical Sciences, University of Sheffield, England}

\date{\today}

\begin{document}

\begin{abstract}
    We introduce a family of homotopy theories for generalized metric spaces with natural number distances, via the magnitude-path spectral sequence (MPSS). The first page of the MPSS is known as magnitude homology; the second page is known as bigraded path homology, and contains GLMY path homology as its top row. For each natural number \(r\), we define a class of maps of metric spaces called \emph{\(r\)-quasi-isomorphisms}: those maps that induce a quasi-isomorphism at page \(r\) of the MPSS. We show that every page of the spectral sequence satisfies a suitable metric analogue of each of the Eilenberg--Steenrod axioms. In particular, we introduce the notion of \emph{\(r\)-cofibration} and prove a Mayer--Vietoris theorem for page \(r\) with respect to \(r\)-cofibrations. We establish a family of Brown category structures on generalized metric spaces which allow us to explicitly compute homotopy colimits. We apply this to describe \emph{\(r\)-suspension} and \emph{\(r\)-spheres of dimension \(n\)}, and compute their spectral sequences. Finally we prove that for \(r=1\) the entire theory restricts to directed graphs.
\end{abstract}

\subjclass[2020]{18G90,   %(2020–now)Other (co)homology theories (category-theoretic aspects)
05C20,    %(1973–now)Directed graphs (digraphs), tournaments
%05C25,  %(1973–now)Graphs and abstract algebra (groups, rings, fields, etc.) 
05C38,    %(1980–now)Paths and cycles 
55U35,    %(1980–now)Abstract and axiomatic homotopy theory in algebraic topology
18G40}     %(1973–now)Spectral sequences, hypercohomology
%18G35   %(1973–now)Chain complexes (category-theoretic aspects), dg categories 

\keywords{Metric spaces, directed graphs, spectral sequences, magnitude-path
spectral sequence, homotopy theory}

\maketitle

\setcounter{tocdepth}{1}
\tableofcontents

%-------------------------------------------------------------
%-------------------------------------------------------------

\section{Introduction}
\label{sec:intro}

As the field of applied topology has evolved, it has become increasingly desirable to be able to speak precisely about the `shapes’ of topologically discrete objects such as graphs and finite metric spaces. This has driven the development of various analogues, for discrete objects, of the tools of classical algebraic topology: homotopy and homology. This paper contributes to an ongoing effort to organize the information captured by such theories and shore up their foundations. Specifically, it concerns a homotopy theory known as \emph{quantitative discrete homotopy theory} and its homological counterpart, the \emph{magnitude-path spectral sequence}. We begin by introducing these terms, before contextualizing the questions addressed in this paper and outlining our main results.

\subsubsection*{Quantitative discrete homotopy theory} 

Traditionally, topologists have tended to view graphs as one-dimensional cell complexes. The theories we consider in this paper represent a shift away from that perspective, to regard some graphs as containing higher-dimensional `voids'. For illustration, consider the directed graph in \Cref{fig:sphere-hasse}, which is the Hasse diagram of a decomposition of the 2-sphere into hemispheres: one might wish to regard it as containing a three-dimensional void. That intuition is captured by a homotopy theory for graphs that we will refer to as \emph{discrete homotopy theory}, following Kapulkin, Carranza and others \cite{Carranza_Kapulkin_2024, KapulkinMavinkurve, EIXYZ}.

Discrete homotopy theory appeared first in the setting of undirected graphs, in work by Babson, Barcelo, de Longueville and Laubenbacher, who called it \emph{A-homotopy theory} \cite{BBLL}. It was extended to directed graphs by Grigor'yan, Lin, Muranov and Yau \cite[\S3.1]{GLMY2014} to describe the homotopy-invariance property of \emph{path homology}, which the same authors introduced in \cite{GLMY2013}. (Path homology is also sometimes referred to as \emph{GLMY homology}.) From the point of view of discrete homotopy theory and path homology, graphs need not be `one-dimensional'. In particular, the path homology of the directed graph in  \Cref{fig:sphere-hasse} coincides with the ordinary homology of the topological 2-sphere \cite[Example~6.17]{GLMY2013}.

\begin{figure}
    \includegraphics[width=.23\textwidth]{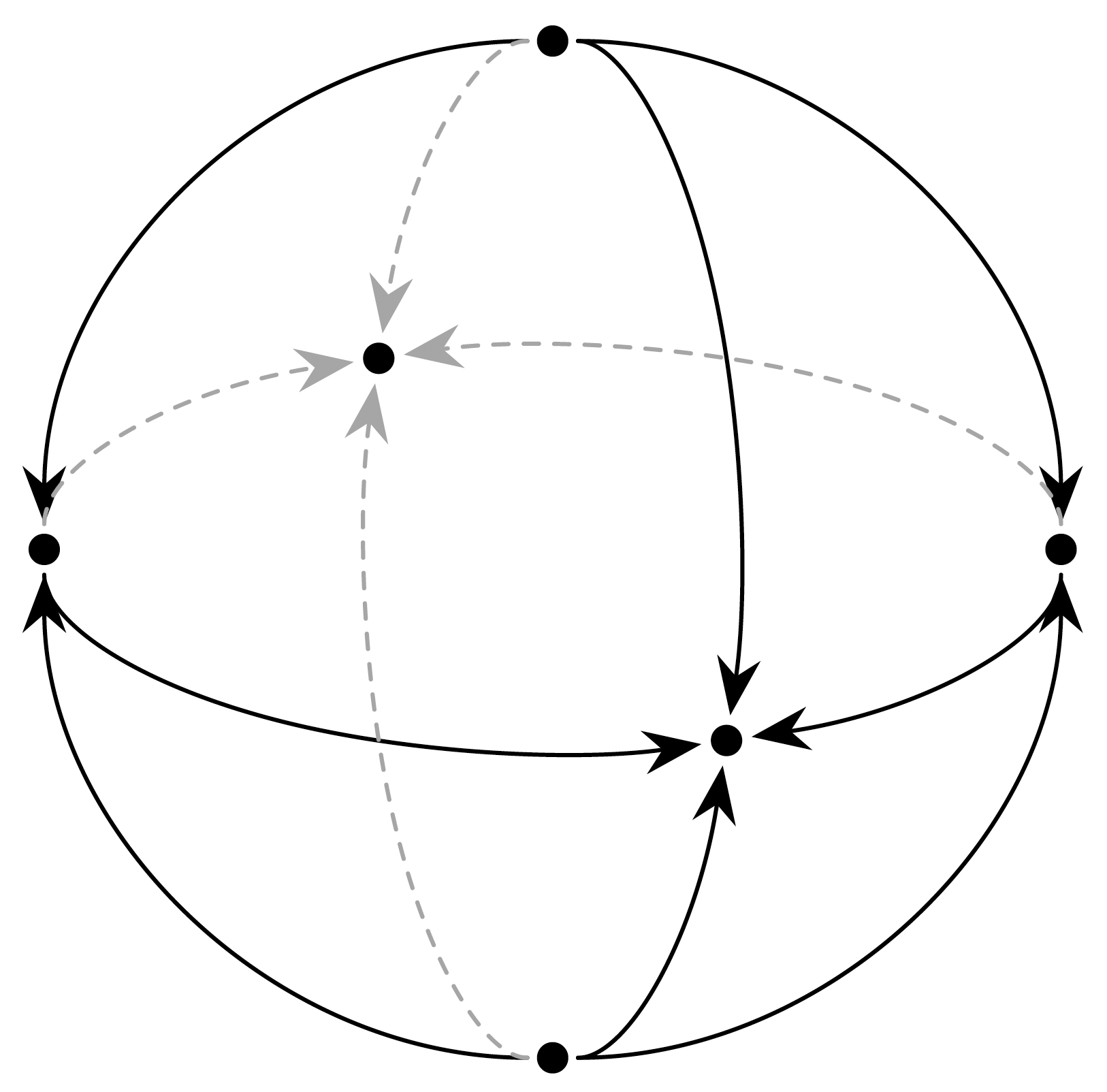}
    \caption{The Hasse diagram of a decomposition of \(S^2\) into hemispheres.}
    \label{fig:sphere-hasse}
\end{figure}

This paper has to do with a modification of discrete homotopy theory that originates in work of Asao \cite{Asao-path, Asao-filtered}. It rests on the fact that every directed graph can be equipped with a distance function, the \emph{shortest-path metric}. This embeds \(\DiGraph\) as a full subcategory of the category \(\Met\) of generalized metric spaces. (In a generalized metric space, distances need not be finite or symmetric.) Consequently, the set of maps between any pair of directed graphs also carries a distance function: the supremum metric. It is straightforward to see that there is a discrete homotopy between maps \(f, g \colon X \rightrightarrows Y\) in \(\DiGraph\) if and only if there exists a finite tuple \((f = f_1, f_2, \ldots, f_k = g)\) of maps from \(X\) to \(Y\) such that, for each \(1 \leq j \leq k-1\), either \(\dis(f_j,f_{j+1}) \leq 1\) or \(\dis(f_{j+1},f_{j}) \leq 1\) in the supremum metric.

From this perspective, discrete homotopy is just the first level in an infinite tower of homotopy theories. One says there is an \emph{\(r\)-homotopy} between \(f\) and \(g\) if there exists a tuple of maps, as above, such that for every \(j\) either \(\dis(f_j,f_{j+1}) \leq r\) or \(\dis(f_{j+1},f_{j}) \leq r\). This defines for each \(r \in \N\) a distinct notion of \emph{\(r\)-homotopy equivalence} for directed graphs. The collection of all these notions of equivalence is what, in a nod to Gromov \cite{Gromov}, we will call \emph{quantitative discrete homotopy theory}.

\subsubsection*{The magnitude-path spectral sequence}

Quantitative discrete homotopy theory extends in a natural way to the category of all generalized metric spaces; indeed, that is arguably the most natural setting in which to study it. In this paper, though, we restrict attention to the subcategory \(\NMet\) whose objects are spaces with metric valued in the natural numbers. This category also contains \(\DiGraph\), and it has the virtue that for every \(r \in \N\) (and any choice of commutative unital ring \(\kk\)) there is a functor from \(\NMet\) to the category of bigraded \(\kk\)-modules that is invariant under \(r\)-homotopies: page \(E^{r+1}\) of the \emph{magnitude-path spectral sequence}, or \emph{MPSS}.

The construction of this spectral sequence was first written down by Hepworth and Willerton in the setting of undirected graphs \cite[Remark 46]{HepworthWillerton2017}. They observed that page \(E^1\) coincides with \emph{magnitude homology}, introduced in the same paper. Later, in the setting of directed graphs, Asao proved that page \(E^2\) contains reduced path homology lying along a single row \cite[Proposition~6.11]{Asao-path}, and observed that the entire page shares path homology's homotopy-invariance \cite[Proposition~5.7]{Asao-path}. This earned the spectral sequence its present name, first used by Di \textit{et al} in \cite{DIMZ}. Since then, the target of the MPSS has been named \emph{reachability homology} \cite{HepworthRoff2023}, and \(E^2\) as a whole \emph{bigraded path homology} \cite{HepworthRoff2024}.

Thus, the magnitude-path spectral sequence encompasses several functors that have been termed `homologies' of metric spaces or graphs. Moreover, just as path homology is a 1-homotopy invariant, for each \(r \geq 1\) page \(E^{r+1}\) of the MPSS is an \(r\)-homotopy invariant. It makes sense, then, to think of the MPSS in its entirety as the homological counterpart to quantitative discrete homotopy theory.

\subsubsection*{Towards a formal framework}

Recently there has been growing interest in placing discrete homotopy within the framework of formal homotopy theory---model categories and related structures---or \(\infty\)-category theory. Both positive and negative results have been obtained in this direction; see, for instance, \cite{CarranzaKapulkinKim2023,Carranza_Kapulkin_2024, EIXYZ}.

Two such results are especially relevant to this paper. In \cite[Theorem 5.1]{CDKOSW}, Carranza \textit{et al} equip \(\DiGraph\) with the structure of a Brown category of cofibrant objects, whose weak equivalences are maps inducing isomorphisms on path homology. (We will recall the definition of a Brown category in \Cref{subsec:brown-cats}.) In \cite[Theorem 7.2]{HepworthRoff2024}, Hepworth and Roff refine that structure to one with the same cofibrations but a smaller class of weak equivalences: maps inducing isomorphisms on bigraded path homology. They pose the question of whether similar structures exist for the later pages of the MPSS. That is, they ask whether \emph{\(r\)-quasi-isomorphisms} in \(\DiGraph\)---maps of graphs inducing isomorphisms on page \(E^{r+1}\)---form the class of weak equivalences in some formal homotopy theory.

In parallel, a separate literature has evolved that appears closely related. Work of Cirici, Livernet, Whitehouse \textit{et al} seeks to construct model categories of filtered complexes \cite{celw}, twisted complexes \cite{celw18}, or multicomplexes \cite{fglw22}, in which the weak equivalences are maps inducing isomorphisms on a specified page of an associated spectral sequence. It seems very natural to try to link these circles of ideas. Indeed, Asao asks in \cite[Remark 4.3.8] {Asao-filtered} whether one can exploit the techniques of \cite{celw} to build a hierarchy of homotopy theories on the category of categories enriched in filtered sets, which contains \(\NMet\) and \(\DiGraph\).

\subsubsection*{Our main results}

This paper aims to advance the understanding of the magnitude-path spectral sequence and of quantitative discrete homotopy theory. In particular, it connects the two literatures mentioned above and offers a partial positive answer to the questions of Hepworth--Roff and Asao.

One of our central contributions is to define, for each natural number \(r \geq 1\), a notion of \emph{\(r\)-cofibration} in \(\NMet\) (\Cref{def:r-cofib}). An \(r\)-cofibration is an especially nice subspace inclusion: one that factors as a \emph{strong \(r\)-deformation retract} followed by an \emph{\(r\)-pair} (Definitions~\ref{def:strong-r-def-retract} and \ref{def:r-pair}). Our main theorem is the following.

\begin{thmNMet-Brown-cat}
    Let \(\Ee_r\) denote the class of \(r\)-quasi-isomorphisms in \(\NMet\) and \(\cof_r\) the class of \(r\)-cofibrations. For each natural number \(r \geq 1\), \((\NMet, \Ee_r, \cof_r)\) is a Brown category of cofibrant objects, with functorial cylinders.
\end{thmNMet-Brown-cat}

On the way to proving \Cref{thm:NMet-Brown-cat}, we establish that every page of the MPSS satisfies a suitable metric analogue of each of the five Eilenberg--Steenrod axioms. Three of these are already known: the \emph{homotopy invariance} property was observed by Asao in \cite{Asao-filtered} and is stated here as \Cref{prop:r-homotopy-invariance}, while the \emph{dimension} and \emph{additivity} properties---which are almost immediate from the construction of the spectral sequence---are stated as Propositions \ref{prop:dimension} and \ref{prop:additivity}. By contrast, the \emph{exactness} of the pages of the MPSS is among the key results of this paper (\Cref{thm:exactness}); it rests on a general observation about maps that may arise only on the first \(r\) pages of a spectral sequence (\Cref{P:partial_filtered}). In \Cref{thm:general-excision} we also prove a new \emph{excision} property for the pages of the MPSS. From exactness and excision together we derive a Mayer--Vietoris sequence on page \(E^r\), with respect to \(r\)-cofibrations (\Cref{thm:mayer-vietoris}).

In a Brown category of cofibrant objects, one has the means to compute many homotopy colimits, including all homotopy pushouts. In \Cref{sec:homotopy-pushouts} we apply this to define an \emph{\(r\)-suspension} functor and construct, for every natural number \(n\) and \(r \geq 1\), an \demph{\(r\)-sphere of dimension \(n\)}, denoted \(\mathbb{S}_r^n\) (\Cref{def:suspension-and-spheres}). Considering two models for the \(r\)-sphere of dimension 1 yields an example illustrating that the class of \(r\)-quasi-isomorphisms is strictly larger than the class of \(r\)-homotopy equivalences (\Cref{prop:E_r_H_r}). In \Cref{thm:spheres} we compute the MPSS for the \(r\)-sphere of dimension \(n\), finding that the bigraded module \(E^{r+1}(\mathbb{S}_r^n)\) is concentrated along a single line, whose slant is determined by \(r\), and on that line it coincides with the ordinary homology of the topological \(n\)-sphere.

\subsubsection*{The structure of the paper} 

\Cref{sec:homotopies} establishes the setting: we describe key features of the categories \(\Met\), \(\NMet\) and \(\DiGraph\), recall the notion of \(r\)-homotopy, and introduce strong \(r\)-deformation retracts. In \Cref{sec:the-MPSS} we collect a few classical facts about filtered complexes and spectral sequences, and prove one new result in that setting, before defining the magnitude-path spectral sequence and describing some of its fundamental properties. \Cref{sec:excision-exactness} contains the exactness and excision theorems. In \Cref{sec:cofibrations-and-MV} we introduce the class of \(r\)-cofibrations and prove the Mayer--Vietoris theorem. The main theorem is proved in \Cref{sec:brown-cats-on-nmet}, and in \Cref{sec:homotopy-pushouts} we apply it to study homotopy pushouts. Finally, in \Cref{sec:brown-cat-on-digraph} we narrow the focus to directed graphs, specializing \Cref{thm:mayer-vietoris} to obtain new Mayer--Vietoris sequences for bigraded path homology and path homology. We verify that in the case \(r=1\) the Brown category structure of \Cref{thm:NMet-Brown-cat} restricts in a non-trivial manner to \(\DiGraph\). The resulting Brown category of directed graphs has the same weak equivalences as in \cite[Theorem 5.1]{HepworthRoff2024}, but a different class of cofibrations.

%-------------------------------------------------------------
%-------------------------------------------------------------

\section{Quantitative discrete homotopy theory}
\label{sec:homotopies}

This section establishes the setting for the paper, beginning with definitions of the relevant categories and some of their basic properties. For each natural number \(r\), the notion of \(r\)-homotopy between maps is explained, together with the associated notion of homotopy equivalence. We define what it means for a subspace to be a strong \(r\)-deformation retract of a larger metric space. This is a variant of a notion due to Ivanov~\cite[\S2.3]{SOIvanov} and it will play an important role in defining \(r\)-cofibrations later. We establish some stability properties for this class of subspace inclusions.

%-------------------------------------------------------------

\subsection{Metric spaces, \(\N\)-metric spaces, and directed graphs}

We introduce the main category studied in this paper, the category \(\NMet\) of \(\N\)-metric spaces, as well as the related categories of directed graphs and (generalized) metric spaces. We describe a closed symmetric monoidal structure on \(\NMet\) and give explicit descriptions of certain well-behaved pushouts and pullbacks.

\begin{defn}\label{def:metric_spaces}
    Let \(\Rbar\) denote the extended real interval \([0, +\infty]\). A \demph{metric space} consists of a set \(X\) and a function \(\dis \colon X \times X \to \Rbar\) satisfying three conditions.
    \begin{enumerate}
        \item \(\dis(x,x) = 0\) for all \(x \in X\). \label{def:unitality}
        \item \(\dis(x,y) = 0\) implies \(x=y\) (\demph{separatedness}). \label{def:separatedness}
        \item \(\dis(x,y) \leq \dis(x,z) + \dis(z,y)\) for all \(x,y,z \in X\) (\demph{the triangle inequality}). \label{def:triangle-ineq}
    \end{enumerate}
    The function \(\dis\) is called the \demph{metric} on \(X\); we will sometimes denote it by \(\dis_X\). We denote by \(\Met\) the category whose objects are metric spaces and in which a morphism \(f \colon (X,\dis_X) \to (Y,\dis_Y)\) is a \demph{1-Lipschitz} map: a function \(f \colon X \to Y\) satisfying
    \[\dis_Y(f(x),f(x')) \leq \dis_X(x,x') \quad \text{for all } x,x' \in X.\]
    We call \((X,\dis)\) an \demph{\(\N\)-metric space} if \(\dis\) is valued in \(\N \cup \{+\infty\} \subset \Rbar\), and denote by \(\NMet\) the full subcategory of \(\Met\) whose objects are \(\N\)-metric spaces.
\end{defn}

\begin{defn}\label{def:isometry}
    A map \(f \colon A \to X\) in \(\Met\) is called a \demph{weak isometry} if it satisfies
    \[\dis_X(f(a),f(a')) = \dis_A(a,a') \quad \text{for all } a,a' \in A \text{ such that } \dis_A(a,a') < \infty.\]
    It is called an \demph{isometry} if it satisfies
    \[\dis_X(f(a),f(a')) = \dis_A(a,a') \quad \text{for all } a,a' \in A.\]
    We will also refer to isometries as \demph{subspace inclusions}.
\end{defn}

To justify this terminology, note that an isometry is always injective: if \(a \neq a'\) then \(0 < \dis_A(a,a') = \dis_X(f(a),f(a'))\), so \(f(a) \neq f(a')\). Thus, a map \(f \colon A \to X\) in \(\Met\) is an isometry if and only if \(A\) is isomorphic to the metric subspace \(f(A)\) of \(X\). Where confusion is unlikely to occur, we will abuse notation slightly by letting \(A\) also denote the subspace \(f(A)\) of \(X\). A bijective isometry is precisely an isomorphism in \(\Met\). 

Our notion of `metric space' differs from the classical one in that a metric need not be symmetric in its variables, and it may not always be finite-valued. This extra generality allows \(\Met\) to accommodate all directed graphs, as we now explain.

\begin{defn}\label{def:digraph}
    A \demph{directed graph} \(G\) is a pair of sets \((V(G), E(G))\) such that \(E(G) \subseteq V(G) \times V(G)\) and for every \(v \in V(G)\) we have \((v,v) \in E(G)\). The elements of \(V(G)\) are called the \demph{vertices} of \(G\) and the elements of \(E(G)\) are its \demph{edges}. We denote by \(\DiGraph\) the category whose objects are directed graphs and in which a morphism \(f \colon G \to H\) is a function \(f \colon V(G) \to V(H)\) satisfying
    \[(u,v) \in E(G) \Rightarrow (f(u),f(v)) \in E(H).\]
    Maps in \(\DiGraph\) will be called \demph{maps of graphs}.
\end{defn}

\begin{defn}\label{def:shortest_path_metric}
    A \demph{directed path} from a vertex \(u\) to a vertex \(v\) in a directed graph \(G\) is a finite sequence of consistently-oriented edges
        \[u = u_0 \to u_1 \to \cdots \to u_k = v.\]
    Note that we do not require the vertices in a directed path to be distinct. Every directed graph \(G\) carries a canonical metric \(\dis_G\) on \(V(G)\), defined by
    \[\dis_G(u,v) = \inf\{k \in \N \mid \text{there is a directed path } u = u_0 \to u_1 \to \cdots \to u_k = v \text{ in } G\}.\]
    This is called the \demph{shortest path metric} on \(G\).
\end{defn}

The shortest path metric need not be symmetric, nor finite-valued. However, it is always separated and it satisfies the triangle inequality, so it makes \(V(G)\) into an \(\N\)-metric space in the sense of \Cref{def:metric_spaces}. We will write \(M(G)\) for \((V(G),\dis_G)\). It is easily verified that a function \(V(G) \to V(H)\) defines a map of graphs \(G \to H\) if and only if it defines a 1-Lipschitz map \(M(G) \to M(H)\). Thus, the assignment \(G \mapsto M(G)\) extends to a functor \(M \colon \DiGraph \hookrightarrow \NMet\) which is full and faithful, exhibiting \(\DiGraph\) as a full subcategory of \(\NMet\). In fact, by the following lemma, it is a coreflective subcategory.

\begin{lem}[{\cite[Lemma~A.1]{HepworthRoff2024}}]\label{lem:coreflective}
    The inclusions 
    \[\DiGraph \xhookrightarrow{M} \NMet \xhookrightarrow{\iota} \Met\]
    each have a right adjoint. \qed
\end{lem}

Explicitly, the right adjoint to \(M\) takes an \(\N\)-metric space \((X,\dis)\) to the directed graph with vertex set \(X\) in which \((x,y)\) is an edge if and only if \(\dis(x,y) \in \{0,1\}\). The right adjoint to \(\iota\) takes a metric space \((X,\dis)\) to the \(\N\)-metric space on \(X\) in which the distance from \(x\) to \(y\) is \(\lceil \dis(x,y) \rceil \coloneqq \min\{n \in \N \mid \dis(x,y) \leq n\}\).

The inclusion of a coreflective subcategory creates all colimits that exist in the ambient category; limits can be formed by taking the limit in the ambient category and applying the right adjoint to the inclusion \cite[Proposition 4.5.15]{Riehl}. With that in mind, we obtain the following explicit descriptions of pullbacks and pushouts along subspace inclusions in \(\NMet\).

\begin{lem}
    \label{lem:NMet_pushout}
  Let \(X\) be an \(\N\)-metric space and \(i \colon A \hookrightarrow X\) a subspace inclusion. Let \(f \colon A \to B\) and \(g \colon U \to X\) be any maps in \(\NMet\). Consider the pullback of \(i\) along \(g\) and the pushout of \(i\) along \(f\), as shown in this diagram:
    \[
    \begin{tikzcd}
        W 
            \arrow[hook,swap]{d}{i_W} 
            \arrow{r}{i^*(g)} 
            \arrow[dr, phantom, "\scalebox{1.5}{$\lrcorner$}" , very near start]
        & A 
            \arrow{r}{f} 
            \arrow[hook]{d}{i}
            \arrow[dr, phantom, "\scalebox{1.5}{\rotatebox{180}{$\lrcorner$}}" , very near end]
        & B 
            \arrow[hook]{d}{i_B} \\
        U 
            \arrow[swap]{r}{g} 
        & X 
            \arrow[swap]{r}{i_*(f)} 
        & Y.
    \end{tikzcd}
    \]
    The pullback \(W\) and the pushout \(Y\) can be described as follows.
    \begin{enumerate}
        \item \label{pt:NMet_pushout_2} \(W\) is the subspace of \(U\) with underlying set \(\{u \in U \mid g(u) \in A\}\).

        \item \label{pt:NMet_pushout_1} As a set, \(Y = B \sqcup X \backslash A\); its metric is given by
        \begin{equation*}
        \dis_Y(y,y') = \begin{cases}
            \dis_B(y,y') &  y,y' \in B \\
            \dis_X(y,y') &  y,y' \in X \setminus A \\
            \min_{a \in A} \{\dis_B(y,f(a)) + \dis_X(a,y')\} & y \in B \text{ and } y' \in X \setminus A \\
            \min_{a \in A} \{\dis_X(y,a) + \dis_B(f(a),y')\} & y \in X \setminus A \text{ and } y' \in B.
        \end{cases}
        \end{equation*}
        \item \label{pt:NMet_pushout_3} Given $\alpha:X\to K$ and
        $\beta:B\to K$ such that \(\beta\circ f=\alpha\circ i\), the unique map \(\rho:Y\to K\) such that \(\rho\circ i_B=\beta\) and \(\rho\circ i_*(f)=\alpha\) is given by
        \[
        \rho(x)=
        \begin{cases}
        \beta(x)& x\in B \\
        \alpha(x)& x\in X\setminus A.
        \end{cases}
        \]
    \end{enumerate}
    The maps \(i_W\) and \(i_B\) are subspace inclusions.
\end{lem}

\begin{proof}
    Part \eqref{pt:NMet_pushout_2} comes easily from the universal property of the pullback. Part \eqref{pt:NMet_pushout_1} and \eqref{pt:NMet_pushout_3}
     can be obtained from the formula for \(\Met\)-colimits given in \cite[Lemma~A.2]{HepworthRoff2024}. That \(i_W\) and \(i_B\) are subspace inclusions is immediate from the descriptions of the pushout and the pullback.
\end{proof}

The category of \(\N\)-metric spaces carries a closed symmetric monoidal structure, which we now describe.

\begin{defn}\label{def:NMet-csmc}
    Let \(\pt\) denote the one-point metric space. For each pair of \(\N\)-metric spaces \((X,\dis_X)\) and \((Y,\dis_Y)\), let $X \plusx Y$ denote the \demph{\(\ell_1\)-product} of \(X\) and \(Y\): the \(\N\)-metric space with underlying set $X \times Y$ and metric
    \begin{equation}\label{eq:plus_times}
        \dis_{X\plusx Y}((x,y),(x',y')) = \dis_X(x,x') + \dis_Y(y,y').
    \end{equation}
    Let $\inhom{X}{Y}$ denote the \(\N\)-metric space with underlying set $\NMet(X,Y)$ and metric
    \begin{equation}\label{eq:sup-hom}
        \dinfty(f,g)=\sup_{x\in X} \dis(f(x),g(x)).
    \end{equation}
\end{defn}

The next statement is well-known for \(\Met\), and the proof for \(\NMet\) is the same.

\begin{thm}\label{thm:NMet-csmc}
    \((\NMet,\plusx,\pt)\) is a closed symmetric monoidal category, with internal hom functor \(\inhom{-}{-}\). In particular we have
    \[\NMet(X \plusx Y, Z) \cong \NMet(X, \inhom{Y}{Z})\]
    for all \(X,Y,Z\) in \(\NMet\). \qed
\end{thm}

%-------------------------------------------------------------

\subsection{Homotopies and homotopy equivalences}

The main theme of this paper is that the category \(\NMet\) carries an infinite family of homotopy theories, one for each natural number \(r\). At the heart of the story is the notion of {\(r\)-homotopy equivalence}, which we will now describe, roughly following Ivanov's presentation in \cite[\S2.2]{SOIvanov}.

All the definitions in this section make equally good sense for the category \(\Met\), by allowing \(r\) to be a positive real number. We restrict our interest to \(\NMet\) because a family of homotopy theories indexed by the natural numbers will fit more easily with the spectral sequence picture developed in the rest of the paper.

\begin{defn}\label{def:r-connected_components}
    Let \((X,\dis)\) be an \(\N\)-metric space. For each \(r \in \N\) we denote by \(\sim_r\) the equivalence relation on \(X\) generated by the relation 
    \[x \sim_r y\Longleftrightarrow \dis(x,y) \leq r.\]
    The set of \demph{\(r\)-connected components} in \(X\) is the quotient set \(X\mod\sim_r\).
\end{defn}

Two maps \(f,g \colon X \rightrightarrows Y\) in \(\NMet\) are said to be \demph{\(r\)-homotopic} if \(f \sim_r g\) in the \(\N\)-metric space \(\inhom{X}{Y}\). It will be helpful to spell out this relation more explicitly.

\begin{defn}\label{def:r-homotopy}
    Fix \(r \in \N\). Let \(f,g \colon (X,\dis_X) \rightrightarrows (Y,\dis_Y)\) be maps in \(\NMet\). We say there is an \demph{elementary \(r\)-homotopy} from \(f\) to \(g\), and write \(f \rightsquigarrow_r g\), if 
    \[\dis_Y(f(x),g(x)) \leq r \; \text{ for every } x \in X.\]
    A \demph{zig-zag} of elementary \(r\)-homotopies between \(f\) and \(g\) is a tuple of maps \(f_1, \ldots, f_k\) in \(\NMet(X,Y)\) such that \(f_1 = f\) and \(f_k = g\) and, for each \(i\), either \(f_i \rightsquigarrow_r f_{i+1}\) or \(f_{i+1} \rightsquigarrow_r f_i\). We say that \(f\) and \(g\) are \demph{\(r\)-homotopic}, and write \(f \sim_r g\), if there exists a zig-zag of \(r\)-homotopies between them.
\end{defn}

Notice that \(\sim_0\) is the equality relation, and that the relation of \(r\)-homotopy grows weaker as \(r\) increases: if \(f \sim_r g\) then \(f \sim_s g\) for every \(s \geq r\).

For every natural number \(r\), \Cref{def:r-homotopy} gives rise to a natural notion of homotopy equivalence for objects of \(\NMet\).

\begin{defn}\label{def:r-homotopy-equiv}
    Fix \(r\in \N\). A map \(f \colon X \to Y\) in \(\NMet\) is an \demph{\(r\)-homotopy equivalence} if there exists a map \(g \colon Y \to X\) such that \(f \circ g \sim_r \id_Y\) and \(g \circ f \sim_r \id_X\). Objects \(X\) and \(Y\) in \(\NMet\) are said to be \demph{\(r\)-homotopy equivalent}, written \(X \simeq_r Y\), if there exists an \(r\)-homotopy equivalence \(f \colon X\to Y\). We denote by \(\mathcal{H}_r\) the class of \(r\)-homotopy equivalences in \(\NMet\).
\end{defn}

The following lemma can be found in Ivanov \cite[Lemma~2.1]{SOIvanov}.

\begin{lem}
\label{lem:whiskering_homotopies}
    Let \(f,g \colon X \rightrightarrows Y\) and \(h \colon W \to X\) and \(k \colon Y \to Z\) be maps in \(\NMet\). For each \(r \in \N\) we have
    \(f \sim_r g \Rightarrow (k\circ f\circ h) \sim_r (k\circ g \circ h).\) \qed
\end{lem}

The next proposition gives some basic information about the
good behaviour of the classes of  \(r\)-homotopy equivalences.

\begin{prop}
    \label{prop:r_htpy_equiv_compose}
    The class  \(\mathcal{H}_0\) is the class of isomorphisms in \(\NMet\) and we have a chain of inclusions
    \[\mathcal{H}_0 \subseteq \mathcal{H}_1 \subseteq \cdots \mathcal{H}_r \subseteq \mathcal{H}_{r+1}\cdots.\] 
    For each \(r \in \N\), the class \(\mathcal{H}_r\) contains all isomorphisms and is closed under composition. 
\end{prop}

\begin{proof}
    Maps \(f, g \colon X \rightrightarrows Y\) satisfy \(f \sim_0 g\) if and only if \(f = g\), so maps in \(\mathcal{H}_0\) are precisely isomorphisms. The chain of inclusions
    follows from the definition.
    That \(\mathcal{H}_r\) is closed under composition follows from \Cref{lem:whiskering_homotopies}. 
\end{proof}

\begin{rem}\label{rem:1-homotopy-discrete-homotopy}
    Using the closed monoidal structure on \(\NMet\), one can give an alternative characterization of the \(r\)-homotopy relation. Define the \demph{\(r\)-interval} to be the two-point space \(I_r = (\{0,1\},\dis)\) in which \(\dis(0,1) = r\) and \(\dis(1,0) = +\infty\). Then, given \(f,g \in \NMet(X,Y)\), there is an elementary \(r\)-homotopy \(f \rightsquigarrow_r g\) if and only if the function \(h \colon I_r \plusx X \to Y\) given by \(h(0,x) = f(x)\) and \(h(1,x) = g(x)\) is 1-Lipschitz. This is how the \(r\)-homotopy relation is presented by Asao \cite[Definition~4.3.5]{Asao-filtered}.

    The 1-interval is in the image of the functor \(M \colon \DiGraph \to \NMet\); it comes from a single edge between two vertices. Using the description in the previous paragraph, one sees that 1-homotopy equivalence for directed graphs corresponds exactly to the form of homotopy equivalence introduced by Grigor'yan, Lin, Muranov and Yau in \cite[\S3]{GLMY2014}. Their notion is in turn a directed variant of the notion of \emph{A-homotopy} equivalence for undirected graphs introduced by Babson, Barcelo, de Longueville and Laubenbacher in \cite[\S2]{BBLL}.
    
    For \(r > 1\), the \(r\)-interval \(I_r\) is not an object of \(\DiGraph\). However, one can consider the class of \(r\)-homotopy equivalences within the full subcategory \(\DiGraph\) of \(\NMet\), as in \cite[Definition 2.2]{HepworthRoff2024}. The directed cycles studied in \cite[\S8]{HepworthRoff2024} illustrate that even within \(\DiGraph\) this notion of equivalence varies with \(r\).
\end{rem}

%-------------------------------------------------------------

\subsection{Strong \(r\)-deformation retracts}

In \cite[\S2.3]{SOIvanov}, Ivanov defines what it means for a metric space \(A\) to be an \demph{\(r\)-deformation retract} of some larger metric space \(X\). We will make use of that notion and a stronger variant.

\begin{defn}
\label{def:strong-r-def-retract}
    Fix \(r \in \N\). Let \(i \colon A \hookrightarrow X\) be a subspace inclusion. We will say that \(A\) is an \demph{\(r\)-deformation retract} of \(X\) if there exists a map \(\pi \colon X \to A\) such that 
    \[\pi \circ i = \id_A \ \text{ and } \ i \circ \pi  \sim_r \id_X.\]
    An \(r\)-deformation retract will be called a \demph{strong \(r\)-deformation retract} if there exists a zig-zag of elementary \(r\)-homotopies \((f_0, \ldots, f_n)\) between \(i \circ \pi\) and \(\id_X\) satisfying
    \[f_\ell \circ i = i\  \text{ and }\ \pi \circ f_\ell = \pi,\ \text{ for each } \ell \in \{0, \ldots, n\}.\]
\end{defn}

In other words, an \(r\)-deformation retract is strong if it can be realized by a zig-zag of \(r\)-homotopies \((f_0,\ldots,f_n)\) in which each map \(f_i\) fixes \(A\) pointwise and preserves the fibres of \(\pi\). Below we see that these conditions ensure the class of strong \(r\)-deformation retracts has good stability properties with respect to composition, pushout and some pullbacks. This will be used in \Cref{sec:brown-cats-on-nmet}.

\begin{lem}
\label{lem:strong_def_compose}
    For each \(r \in \N\), the class of strong \(r\)-deformation retracts is stable under composition.
\end{lem}

\begin{proof}
    Suppose \(i_1 \colon X \rightleftarrows Y : \pi_1\) and \(i_2 \colon Y \rightleftarrows Z : \pi_2\) are strong \(r\)-deformation retracts. Choose zig-zags of \(r\)-homotopies \((f_0,\ldots, f_m)\) between \(i_1 \circ \pi_1\) and \(\id_Y\), and \((g_0, \ldots, g_n)\) between \(i_2 \circ \pi_2\) and \(\id_Z\), both satisfying \Cref{def:strong-r-def-retract}. For \(0 \leq \ell \leq m\), let \(h_\ell = i_2 \circ f_\ell \circ \pi_2\). By \Cref{lem:whiskering_homotopies}, the tuple \((h_0, \ldots, h_m)\) is a zig-zag of \(r\)-homotopies between \(i_2 \circ i_1 \circ \pi_1 \circ \pi_2\) and \(i_2 \circ \pi_2\), so we can concatenate it with \((g_0, \ldots, g_n)\) to obtain one between \(i_2 \circ i_1 \circ \pi_1 \circ \pi_2\) and \(\id_Z\).
    
    To see that this zig-zag satisfies \Cref{def:strong-r-def-retract}, note that \(g_\ell \circ (i_2 \circ i_1) = i_2 \circ i_1 \) for each \(\ell\), thanks to the assumption that \(g_\ell \circ i_2 = i_2\); similarly, \((\pi_1 \circ \pi_2) \circ g_\ell = \pi_1 \circ \pi_2\). Meanwhile, we have
    \begin{align*}
        h_\ell \circ (i_2 \circ i_1) =& i_2 \circ f_\ell \circ \pi_2 \circ i_2 \circ i_1 = i_2 \circ f_\ell \circ i_1 = i_2 \circ i_1,\\
    (\pi_1 \circ \pi_2) \circ h_\ell =& \pi_1 \circ \pi_2 \circ i_2 \circ f_\ell \circ \pi_2 = \pi_1 \circ f_\ell \circ \pi_2 = \pi_1 \circ \pi_2
    \end{align*}
    since \(\pi_2 \circ i_2 = \id_Y\), \(f_\ell \circ i_1 = i_1\) and \(\pi_1 \circ f_\ell = \pi_1\) by assumption. This proves that \(i_2 \circ i_1 \colon X \rightleftarrows Z : \pi_1 \circ \pi_2\) is a strong \(r\)-deformation retract.
\end{proof}

\begin{prop}
\label{prop:strong_def_pushout}
    For each \(r \in \N\), the class of strong \(r\)-deformation retracts is stable under pushout along any map.
\end{prop}

\begin{proof}
    Let \(i \colon A \rightleftarrows U : \pi\) be a strong \(r\)-deformation retract and let \(f \colon A \to B\) be any map in \(\NMet\). Let \(j \colon B \to W \leftarrow U : g\) be the pushout of \(i\) along \(f\). We show that there exists \(p\) such that \(j \colon B \rightleftarrows W : p\) is a strong \(r\)-deformation retract.
    By \Cref{lem:NMet_pushout} the map \(j\) is a subspace inclusion. Let \((h_0,\ldots, h_n)\) be a zig-zag that exhibits the \(r\)-homotopy between \(i \circ \pi\) and \(\id_U\) and satisfies \(h_\ell \circ i = i\) and \(\pi \circ h_\ell = \pi\) for each \(\ell\). By considering the universal property of the pushout in the diagrams
    \begin{equation}
    \label{diag:pushout_1}
    \begin{tikzcd}
        A \arrow[hook]{r}{i} \arrow[swap]{d}{f} \arrow[dr, phantom, "\scalebox{1.5}{\rotatebox{180}{$\lrcorner$}}" , very near end] & U \arrow[]{d}{g} 
        \arrow{r}{\pi} & A\arrow{d}{f} \\
        B \arrow[hook]{r}{j} \arrow[bend right=30, swap]{rr}{\id_B} & W \arrow[dashed]{r}{\exists ! p} & B
       \end{tikzcd} \hspace{1cm} 
    \begin{tikzcd}
        A \arrow[hook]{r}{i} \arrow[swap]{d}{f} \arrow[dr, phantom, "\scalebox{1.5}{\rotatebox{180}{$\lrcorner$}}" , very near end] & U \arrow[]{d}{g} 
        \arrow{r}{h_\ell} & U\arrow{d}{g} \\
        B \arrow[hook]{r}{j} \arrow[bend right=25, swap]{rr}{j} & W \arrow[dashed]{r}{\exists ! k_l} & W
       \end{tikzcd} 
    \end{equation}
    we obtain a unique map \(p \colon W \to B\) satisfying \(p \circ j = \id_B\) and \(p\circ g=f\circ\pi\) and, for each \(\ell \in \{0,\ldots, n\}\), a unique map \(k_\ell \colon W \to W\) satisfying \(k_\ell \circ j = j\) and \(k_\ell\circ g=h_\ell\circ g\). Note that the relations \(h_0=i\circ\pi\), \(h_n=\id_U\) and \(\pi\circ h_\ell=\pi\) for each \(\ell\) induce the relations \(k_0=j\circ p\), \(k_n=\id_W\) and  \(p\circ k_\ell=p\) for each \(\ell\). It remains to show that \((k_0,\ldots,k_n)\) is a zig-zag of \(r\)-homotopies. For this, it is enough to prove that if \(h_\ell \rightsquigarrow_r h_{\ell+1}\) then \(k_\ell \rightsquigarrow_r k_{\ell+1}\). Assume \(\dis_U(h_\ell(u),h_{\ell+1}(u)) \leq r\) for all \(u \in U\).  By \Cref{lem:NMet_pushout} we have \(W = B \sqcup U \setminus A\) as sets, and
    \[
    k_\ell(w) = 
    \begin{cases}
        w &  w \in B \\
        g (h_\ell(w)) & w \in U \setminus A.
    \end{cases}
    \]
    Thus, if  \(w \in B\) then \(\dis_W(k_\ell(w), k_{\ell+1}(w)) = \dis_B(w,w) = 0\). If \(w\in U\setminus A\)
    \[\dis_W(k_l(w),k_{l+1}(w))=\dis_W(g(h_l(w)),g(h_{l+1}(w))\leq \dis_U(h_l(w),h_{l+1}(w))\leq r.\]
    In conclusion, \(j \colon B \rightleftarrows W : p\) is a strong \(r\)-deformation retract with the zig-zag \((k_0,\ldots,k_n)\) exhibiting the strong \(r\)-homotopy between \(j\circ p\) and \(\id_W\).
\end{proof}

\begin{lem}
\label{lem:strong_def_pullback}
    Let \(j_1 \colon X \rightleftarrows V : \pi_V\) be a strong \(r\)-deformation retract and \(i_2 \colon U \hookrightarrow X\) a subspace inclusion. Let \(U \xleftarrow{\pi_W} W \xhookrightarrow{i_W} V\) be the pullback of \(\pi_V\) along \(i_2\). There exists \(k\) such that \(k \colon U \rightleftarrows W : \pi_W\) is a strong \(r\)-deformation retract.
\end{lem}

\begin{proof}
    By \Cref{lem:NMet_pushout}~\!\eqref{pt:NMet_pushout_2}, \(W\) is the subspace of \(V\) whose set of points is \(\{v \in V \mid \pi_V(v) \in U\}\); the map \(\pi_W\) is the restriction of \(\pi_V\) to \(W\).   The assumption that \(X\) is a strong \(r\)-deformation retract of \(V\) means we have a zig-zag \((f_0,\ldots, f_n)\) of \(r\)-homotopies between \(f_0=j_1 \circ \pi_V\) and \(f_n=\id_V\), such that \(f_\ell \circ j_1 = j_1\) and \(\pi_V \circ f_\ell = \pi_V\) for each \(\ell \in \{0,\ldots, n\}\). By considering the diagrams
    \begin{equation*}
    \label{diag:pullback_1}
    \begin{tikzcd}
        U \arrow[hook]{rr}{i_2} \arrow[bend right=30, swap]{ddr}{\id_U} \arrow[dashed]{dr}{\exists ! k} && X \arrow[hook]{d}{j_1} \arrow[bend left=60]{dd}{\id_X} \\
        & W \arrow[hook]{r}{i_W} \arrow[swap]{d}{\pi_W} 
        \arrow[dr, phantom, "\scalebox{1.5}{$\lrcorner$}" , very near start]
        & V \arrow{d}{\pi_V} \\
        & U \arrow[hook, swap]{r}{i_2} & X
    \end{tikzcd} \hspace{1cm} \begin{tikzcd}
        W \arrow[hook]{rr}{i_W} \arrow[bend right=30, swap]{ddr}{\pi_W} \arrow[dashed]{dr}{\exists ! h_l} && V \arrow[hook]{d}{f_l} \arrow[bend left=60]{dd}{\pi_V} \\
        & W \arrow[hook]{r}{i_W} \arrow[swap]{d}{\pi_W} 
        \arrow[dr, phantom, "\scalebox{1.5}{$\lrcorner$}" , very near start]
        & V \arrow{d}{\pi_V} \\
        & U \arrow[hook, swap]{r}{i_2} & X
    \end{tikzcd}
    \end{equation*}
    we see that the universal property of the pullback gives a unique map \(k \colon U \hookrightarrow W\)  satisfying \(\pi_W \circ k = \id_U\) and  \(i_W\circ k=j_1\circ i_2\) and for each \(\ell \in \{0,\ldots, n\}\), a unique map  \(h_\ell:W\to W\) satisfying \(\pi_W \circ h_\ell = \pi_W\) and \( i_W\circ h_\ell=f_\ell\circ i_W\). By uniqueness in the pullback diagrams we have \(h_\ell \circ k=k, h_0=k\circ\pi_W \) and \(h_n=\id_W\).
    
    Since \(f_\ell \circ i_W = i_W\circ h_\ell\), each \(h_\ell\) is the restriction of \(f_\ell\) to \(W\), so if \(f_\ell \rightsquigarrow_r f_{\ell+1}\) then \(h_\ell \rightsquigarrow_r h_{\ell+1}\) and if \(f_\ell \leftsquigarrow_r f_{\ell+1}\) then \(h_\ell \leftsquigarrow_r h_{\ell+1}\). This tells us the tuple \((h_0,\ldots, h_n)\) is a zig-zag of \(r\)-homotopies between \(k \circ \pi_W\) and \(\id_W\), and hence that \(k \colon B \rightleftarrows W : \pi_W\) is a strong \(r\)-deformation retract.
\end{proof}

%-------------------------------------------------------------
%-------------------------------------------------------------

\section{The magnitude-path spectral sequence}
\label{sec:the-MPSS}

In this section we will see that every \(\N\)-metric space gives rise to a filtered chain complex, known as the {reachability complex}, and thereby to a spectral sequence, the {magnitude-path spectral sequence}. We will introduce the class of {\(r\)-quasi-isomorphisms} of \(\N\)-metric spaces and note its relationship to the class of \(r\)-homotopy equivalences.

%-------------------------------------------------------------

\subsection{Filtered complexes and spectral sequences} 
\label{subsec:fcx}

We refer to Cartan--Eilenberg \cite{CartanEilenberg56} for the general theory of filtered complexes and spectral sequences, recalling here only what is needed for our purposes. We also prove a new result, \Cref{P:partial_filtered}, which will be used in the sequel and is of independent interest.

Throughout the paper, \(\kk\) will denote a fixed commutative ring with unit \(1_\kk\). In this section, graded objects are graded by \(\Z\).

\begin{defn}
    A \emph{filtered graded \(\kk\)-module} \((R,F)\) is a graded \(\kk\)-module \(R\) with sub-graded \(\kk\)-modules \(F_pR\) for \(p\in\Z\) satisfying \(F_pR\subset F_{p+1}R\), for each \(p\in\Z\). 
    
    A \emph{filtered complex} \((C,\partial,F)\) is both a chain complex \((C,\partial)\) and a filtered graded module \((C,F)\) satisfying \(\partial(F_pC)\subset F_pC\), for each \(p\in\Z\).
    
    A \emph{morphism of filtered complexes} \(f \colon (C,\partial,F) \to (C',\partial',F')\) is a morphism of chain complexes satisfying \(f(F_pC)\subset F'_pC'\), for each \(p\in\Z\). Such a morphism is said to be \emph{strict} if \(f(F_p(C))=f(C)\cap F'_p(C')\), for each \(p\in\Z\).
    
    We denote by \(\mathrm{gMod}_\kk\) the category of graded \(\kk\)-modules, and by \(\fc\) the category of filtered complexes over \(\kk\).
\end{defn}

\begin{defn}
    Let \(r\geq 0\) be an integer. An \emph{\(r\)-bigraded complex} \(E\) is a bigraded \(\kk\)-module endowed with a morphism \(\Delta^r \colon E_{\ast,\ast}\to E_{\ast-r,\ast+r-1}\) such that \((\Delta^r)^2=0\).

    A \emph{spectral sequence} \(E\) is a collection of \(r\)-bigraded complexes \((E^r,\Delta^r)_{r\geq 0}\) together with a collection of isomorphisms \(\varphi_r \colon H_*(E^r,\Delta^r)\to E^{r+1}\) of bigraded \(\kk\)-modules, called the characteristic maps.

    A \emph{morphism of spectral sequences} \(f \colon E\to E'\) is a collection of morphisms \(f_r \colon E^r\to {E'}^r \) of \(r\)-bigraded complexes compatible with the characteristic maps.

    We denote by \(\bgrmod\) the category of bigraded \(\kk\)-modules, and by \(\spse\) the category of spectral sequences.
\end{defn}

To every filtered complex \((C,\partial,F)\) one can associate a spectral sequence \((E^r(C),\Delta^r)_{r\geq 0}\), and this assignment is functorial. We denote by \(E^{\bullet}:\fc\to\spse\) the resulting functor. We recall here one description of the pages of the spectral sequence, given in  \cite[Section 1.3.1]{Deligne}. Let \((C,\partial,F)\) be a filtered complex. Its associated spectral sequence has the following form:
    \[E^r_{p,q}(C)=\texttt{Z}^r_{p,q}(C)/\texttt{B}^r_{p,q}(C)\]
with  \(Z^0_{p,q}(C)=F_pC_{p+q},\ B^0_{p,q}(C)=F_{p-1}(C_{p+q})\) and, for \(r\geq 1\),
    \[\texttt{Z}^r_{p,q}(C)=F_pC_{p+q}\cap \partial^{-1}(F_{p-r}C_{p+q-1}) \text{ and }
    \texttt{B}^r_{p,q}(C)=\texttt{Z}^{r-1}_{p-1,q+1}(C)+\partial\texttt{Z}^{r-1}_{p+r-1,q+2-r}(C).\]
The differential \(\Delta^r\) is given for \([x]\in E^r_{p,q}(C)\) by \(\Delta^r([x])=[\partial x]\).

\medskip

Next we recall notions of homotopy for filtered complexes and the homotopy invariance property of the associated spectral sequence.

\begin{defn}\label{D:r-homotopic_filtered}
    Fix \(r \in \N\). Given two morphisms \(f,g \colon (C,\partial,F) \to (C',\partial',F')\) in \(\fc\), a \demph{homotopy of order \(r\)} from \(f\) to \(g\) is a map of graded \(\kk\)-modules \(h \colon C_*\to C'_{*+1}\) satisfying \(\partial'h+h\partial=f-g\) and \(h(F_pC)\subset F_{p+r}C'\) for every \(p\in\Z\). If such a map exists, we say that \(f\) and \(g\) are \demph{\(r\)-homotopic}.
\end{defn}

\begin{prop}{\cite[Proposition XV.3.1]{CartanEilenberg56}} \label{P:CE_invariance}
\label{prop:EM-r-homotopy}
    If \(f\) and \(g\) are \(r\)-homotopic, then for all \(s>r\) we have \(E^s(f)=E^s(g).\) \qed
\end{prop}

Whereas \Cref{prop:EM-r-homotopy} concerns the \emph{eventual} behaviour of a spectral sequence, the next proposition concerns its \emph{initial} behaviour. It establishes a condition for a map of underlying filtered graded \(\kk\)-modules between two filtered complexes to induce a map on the first \(r\) pages of the associated spectral sequences. This will be used to prove one of the key results of the paper, \Cref{thm:exactness}.

\begin{prop}\label{P:partial_filtered}
    Fix \(r\geq 0\). Let \((C,\partial,F)\) and \((D,\partial',F')\) be two filtered complexes and \(f \colon C\to D\) be a morphism of filtered graded \(\kk\)-modules. Suppose that for all \(x\in F_{\ell}C\) we have
    \[(\partial'f-f\partial)(x)\in F_{{\ell}-r}D.\]
    Then \(f\) induces a morphism of \(s\)-bigraded complexes \(E^s(C)\to E^s(D)\) for \(0\leq s<r\) and a morphism of bigraded \(\kk\)-modules \(E^r(C)\to E^r(D).\)
\end{prop}

\begin{proof}
    We first prove that \(f(\texttt{Z}^s_{p,q}(C))\subset \texttt{Z}^s_{p,q}(D)\) for \(0\leq s\leq r\). Let \(x\in F_pC_{p+q}\) such that \(\partial x\in F_{p-s}C_{p+q-1}\). We have \(f(x)\in F_pD_{p+q}\) and \(\partial'f(x)=w+f\partial(x)\) for some \(w\in F_{p-r}D_{p+q-1}\subset F_{p-s}D_{p+q-1}\). Since \(f(\partial x  )\in F_{p-s}D_{p+q-1}\) we get the result.
    
    Let us prove that \(f(\texttt{B}^s_{p,q}(C))\subset \texttt{B}^s_{p,q}(D)\) for \(0\leq s\leq r\). If \(s=0\) this is clear. Assume \(s\geq 1\). It is enough to prove that \(f(\partial \texttt{Z}^{s-1}_{p+s-1,q+2-s}(C))\subset \texttt{B}^s_{p,q}(D)\) because we have proved that \(f(\texttt{Z}^{s-1}_{p-1,q+1}(C))\subset \texttt{Z}^{s-1}_{p-1,q+1}(D)\). Let \(x=\partial y\) with \(y\in \texttt{Z}^{s-1}_{p+s-1,q+2-s}(C)\). In particular, \(x=\partial y\in F_pC.\) By the first part of the proof we have \(f(y)\in \texttt{Z}^{s-1}_{p+s-1,q+2-s}(D)\). We have
    \[f(x)=f(\partial y)=u+\partial 'f(y)\] 
    for some \(u\in F_{p+s-1-r}D_{p+q}\subset F_{p-1}D_{p+q}\). Then
    \(\partial 'u=\partial 'f(\partial y)=f(\partial ^2y)+u'=u'\)
    for some \(u'\in F_{p-r}D_{p+q-1}\subset F_{p-s}D_{p+q-1}.\) As a consequence \(u\in \texttt{Z}^{s-1}_{p-1,q+1}(D)\) and \(f(x)\in \texttt{B}^s_{p,q}(D)\).

    Finally, let \(x\in \texttt{Z}^s_{p,q}(C)\), so \(\partial x\in F_{p-s}C_{p+q}\), and let \(s<r\). We prove that  \(\partial 'f(x)-f\partial (x)\in \texttt{B}^s_{p-s,q+s-1}(D)\). We have \(\partial' f(x)-f\partial( x)=v\) for some \(v\) in \(F_{p-r}D\subset F_{p-s-1}D_{p+q-1}\), and
    \[\partial' v=-\partial' f\partial x=-f\partial^2x+v'\]
    for some \(v'\) in \(F_{p-s-r}D_{p+q-2}\subset F_{p-2s}D_{p+q-2}\), so \(v\) is in \(\texttt{Z}^{s-1}_{p-s-1,q+s}(D)\). This implies that \(v\in \texttt{B}^s_{p-s,q+s-1}(D)\).
\end{proof}

In the rest of the paper, we will be interested in the special case of filtered complexes in which both the chain complex and the filtration is non-negatively graded. 

%-------------------------------------------------------------

\subsection{The reachability complex and the magnitude-path spectral sequence}
\label{subsec:MPSS}

Next, we describe the {reachability complex} of an \(\N\)-metric space, its {length filtration}, and the associated {magnitude-path spectral sequence}. After recalling the definitions, we prove a new result, \Cref{prop:E^r_00}, describing how page \(E^{r+1}\) of the magnitude-path spectral sequence records the number of \(r\)-connected components in an \(\N\)-metric space.

\begin{defn}
\label{def:RC-chain-complex}
     The \demph{reachability complex} of an \(\N\)-metric space \((X, \dis )\) is the chain complex \((\RC(X), \partial)\) of \(\kk\)-modules which is freely generated in degree \(k\) by tuples \((x_0, \ldots, x_{k})\) of points in \(X\) such that
    \begin{enumerate}
        \item \label{cond:RC-1} consecutive terms are distinct: \(x_{j-1}\neq x_j\) for \(j\in \{1,\ldots, k\}\), and
        \item \label{cond:RC-2} each term can \demph{reach} the next: \(\dis(x_{j-1},x_j) < \infty\) for each \(j \in \{1,\ldots, k\}\).
    \end{enumerate}
    The differential \(\partial \colon \RC_\ast(X) \to \RC_{\ast-1}(X)\) is determined on generators by discarding each term of the tuple in turn and taking the alternating sum:
    \[\partial((x_0,\ldots, x_k)) = \sum_{j=0}^{k} (-1)^j (x_0,\ldots, \widehat{x_j}, \ldots, x_k)\ \text{ where}\]
    \[(x_0,\ldots, \widehat{x_j}, \ldots, x_k) = \begin{cases}
        (x_0,\ldots, x_{j-1}, x_{j+1}, \ldots, x_k) &  x_{j-1} \neq x_{j+1} \\
        0 & x_{j-1} = x_{j+1}.
    \end{cases}\]
\end{defn}

Notice that condition \eqref{cond:RC-2} in \Cref{def:RC-chain-complex} is equivalent to requiring that the \demph{length} of each generating tuple \((x_0,\ldots,x_k)\) is finite, meaning \(\sum_{i=1}^k \dis(x_{i-1}, x_i) < \infty\). It follows that the reachability complex \(\RC(X)\) has a natural filtration by the lengths of its generators. Indeed, if we define the \(\ell^\mathrm{th}\) piece of the filtration to be
\[
    F_\ell \RC_{k}(X) = \kk \cdot \left\{(x_0,\ldots, x_k) \mid x_{j-1} \neq x_{j} \text{ for all \(j\) and } \sum_{j=1}^k \dis(x_{j-1},x_j) \leq \ell\right\}
\]
then we have that \(\partial(F_\ell \RC_{k}(X)) \subseteq F_\ell \RC_{k-1}(X)\) because, by the triangle inequality, removing a term from a generating tuple can only reduce its length. This construction extends to a functor \(\RC\colon \NMet\to \fc\).

\begin{rem}
By splitting according to length, one sees that the reachability complex with the length filtration can be viewed as a \emph{split filtered complex} in the
sense of~\cite[Definition 3.8]{celw18}. That is, it is the image under the
totalization functor \(\Tot\) of a \emph{twisted complex}. We have chosen to work with filtered complexes in this paper, but one can also formulate several results and proofs in terms of twisted complexes and we have found that perspective useful.
\end{rem}

\begin{defn}
\label{def:mpss}
    The \demph{magnitude-path spectral sequence} (or \demph{MPSS}) of an \(\N\)-metric space \(X\) is the spectral sequence associated to the length filtration of \(\RC(X)\). We will denote it by \((E^r(X), \Delta^r)_{r \geq 0}\).
\end{defn}

Since the spectral sequence construction is functorial with respect to maps of filtered complexes, the magnitude-path spectral sequence as a whole defines a functor
\[E^{\bullet} \colon \NMet \to \spse\]
and, for each \(r \geq 0\), page \(E^r\) defines a functor
\[E^r \colon \NMet \to \bgrmod.\]

As discussed in the introduction, page \(E^1\) of the MPSS is the restriction to \(\NMet\) of the functor \(\MH \colon \Met \to \bgrmod\) known as \emph{magnitude homology}: for each \((p,q)\) there is a natural isomorphism \(E^1_{p,q}(X) \cong \MH_{p+q,p}(X)\). Meanwhile, for a directed graph \(G\), there is an isomorphism \(E^2_{n,0}(M(G)) \cong \GLMY_n(G)\), where \(M \colon \DiGraph \hookrightarrow \NMet\) equips a directed graph with its shortest path metric and \(\GLMY \colon \DiGraph \to \mathrm{gMod}_\kk\) is \emph{reduced path homology} \cite[Proposition 6.11]{Asao-path}.

Both magnitude homology and path homology are well-studied invariants in their own right, and each has been shown to capture interesting geometric or combinatorial information---for example about curvature \cite{Asao-curvature, PH-curvature}, convexity \cite{LeinsterShulman}, uniqueness of geodesics \cite{GomiGeodesic} and girth \cite{Asao-curvature}. So far, though, almost nothing is known about the geometric or combinatorial information captured by the later pages of the MPSS. We close this section with a first observation in that direction: in bigrading \((0,0)\), page \(E^{r+1}\) counts \(r\)-connected components (\Cref{def:r-connected_components}).

\begin{prop}\label{prop:E^r_00}
    Let \(X\) be an \(\N\)-metric space. For each \(r \geq 0\), the \(\kk\)-module \(E^{r+1}_{0,0}(X)\) is freely generated by the set of \(r\)-connected components of \(X\).
\end{prop}

\begin{proof}
    In this proof we set \(\tRC_{p,q}(X)=E^0_{p,q}(\RC(X))\). This is the free \(\kk\)-module spanned by elements in \(\RC_{p+q}(X)\) of the form \((x_0, \ldots, x_{p+q})\) of length \(p\).

    Let \(C=\RC(X)\) be the reachability complex of \(X\), considered as a filtered complex. From the description of the spectral sequence in~\Cref{subsec:fcx} we have \(\texttt{Z}^{r+1}_{0,0}(C)=\tRC_{0,0}(X)\) and \(\texttt{B}^{r+1}_{0,0}(C)=\partial\texttt{Z}^r_{r,1-r}(C)\) with 
    \[\texttt{Z}^r_{r,1-r}(C)=\bigoplus_{k=1}^{r} \tRC_{k,1-k}(X)\]
    so that
    \[E^{r+1}_{0,0}(X) \cong \mathrm{coker} \left( \partial \colon \bigoplus_{k=1}^{r} \tRC_{k,1-k}(X) \to \tRC_{0,0}(X) \right)\]
    with \(\partial(x_0,x_1)=(x_1)-(x_0)\).

    Denote by \(W_r\) the free \(\kk\)-module generated by the set \(X\mod\sim_r\) of \(r\)-connected components of \(X\). We will prove that \(W_r\) satisfies the universal property of the cokernel.
    
    Let \(\pi \colon \tRC_{0,0}(X)\to W_r\) be the \(\kk\)-linear map induced by the quotient \(X\to X\mod\sim_r\). Clearly, \(\pi \partial=0\). Now let \(u \colon \tRC_{0,0}(X)\to P\) be any \(\kk\)-linear map satisfying \(u\partial=0\). Given any \(x \neq y\) in \(X\) such that \(x\sim_r y\), there exist \(x_0=x,x_1,\ldots, x_{N-1},x_N=y\) in \(X\) such that \(0<\dis_X(x_i,x_{i+1})\leq r\) or \(0<\dis_X(x_{i+1},x_i)\leq r\) for \(0 \leq i < N.\) Let \(c \in \bigoplus_{k=1}^{r} \tRC_{k,1-k}(X)\) be defined by \(c=\sum_{i=0}^{N-1} c_i\) where 
    \[c_i=\begin{cases}
        (x_i,x_{i+1}) &  \dis_X(x_i,x_{i+1})\leq r\\
        -(x_{i+1}, x_i) &  \dis_X(x_{i+1},x_{i}) \leq r \text{ and } \dis_X(x_i,x_{i+1})> r.
    \end{cases} 
    \]
    Then \(\partial( c) = y-x\). This proves that \(u \colon \tRC_{0,0}(X)\to P\) uniquely defines a morphism \(\bar u \colon W_r\to P\) such that \(\bar u\pi=u.\) Hence \(E^r_{0,0}(X)\cong W_r.\)   
\end{proof}

%-------------------------------------------------------------

\subsection{Fundamental properties of the pages of the MPSS}
\label{subsec:MPSS-properties}

We have seen that each page of the magnitude-path spectral sequence defines a functor from \(\NMet\) to the category of bigraded \(\kk\)-modules. We now record three more fundamental properties of the pages, analogous to three of the Eilenberg--Steenrod axioms for the ordinary homology of topological spaces: dimension, additivity, and homotopy invariance. In \Cref{sec:excision-exactness}, we will prove that each page satisfies analogues of the two remaining axioms: exactness and excision. All five properties will be used later in the paper.

The dimension and additivity properties follow almost immediately from the definition, but so far as we know their statements do not yet appear in the literature.

\begin{prop}[Dimension Property]
\label{prop:dimension}
    Let \(\pt\) denote the one-point \(\N\)-metric space. For each \(r \geq 0\), we have \(E_{0,0}^r(\pt) = \kk\) while \(E_{p,q}^r(\pt) = 0\) for \((p,q) \neq (0,0)\).
\end{prop}

\begin{proof}
    Since the reachability complex \(\RC(\pt)\) is concentrated in degree $0$ with constant filtration we have \(E_{0,0}^0(\pt) = \kk\) and \(E_{p,q}^0(\pt) = 0\) whenever \((p,q) \neq (0,0)\) with differential zero, and the statement follows.
\end{proof}

\begin{prop}[Additivity]
\label{prop:additivity}
    Given \(\N\)-metric spaces \(X\) and \(Y\), let \(X \sqcup Y\) denote their coproduct in \(\NMet\). For each \(r \geq 0\) we have
    \begin{equation*}\label{eq:additivity}
        E^r(X \sqcup Y) \cong E^r(X) \oplus E^r(Y).
    \end{equation*}
\end{prop}

\begin{proof}
    By inspection, \(\RC(X \sqcup Y) \cong \RC(X) \oplus \RC(Y)\) as filtered chain complexes.
\end{proof}

The dimension and additivity statements each hold universally for the pages of the MPSS. By contrast, each page \(E^r\) has its own distinct homotopy-invariance property. Concretely, if \(f, g \colon X \rightrightarrows Y\) are maps in \(\NMet\) that are \(r\)-homotopic in the sense of \Cref{def:r-homotopy}, then \(\RC(f)\) and \(\RC(g)\) are \(r\)-homotopic as maps of filtered chain complexes. By \Cref{P:CE_invariance}, this implies the following statement, first observed by Asao in \cite[Theorem 4.37]{Asao-filtered}.

\begin{prop}[Homotopy Invariance]
\label{prop:r-homotopy-invariance}
    If \(f, g \colon X \rightrightarrows Y\) are \(r\)-homotopic maps of \(\N\)-metric spaces, then \(E^s(f) = E^s(g)\) for all \(s > r\). Thus, if \(f \colon X \to Y\) is an \(r\)-homotopy equivalence, then \(E^r(f) \colon E^r(X) \to E^r(Y)\) is a quasi-isomorphism and \(E^s(f) \colon E^s(X) \to E^s(Y)\) is an isomorphism for all \(s > r\). \qed
\end{prop}

\Cref{prop:r-homotopy-invariance} says that the homotopy invariance of the MPSS changes as one turns the pages of the sequence. Indeed, page \(E^{r+1}\) is an \(r\)-homotopy invariant but \emph{not} an \((r+1)\)-homotopy invariant, as \cite[Corollary 8.3]{HepworthRoff2024} shows.

In later sections of the paper, the following class of maps will play a central role.

\begin{defn}\label{def:r-quasi-iso}
    Fix \(r \in \N\). We will call a map \(f \colon X \to Y\) in \(\NMet\) an \demph{\(r\)-quasi-isomorphism} if the induced map \(E^{r}(f) \colon E^{r}(X) \to E^{r}(Y)\) is a quasi-isomorphism. We will denote by \(\Ee_r\) the class of \(r\)-quasi-isomorphisms in \(\NMet\).
\end{defn}

By \Cref{prop:r-homotopy-invariance}, we have \(\mathcal{H}_r \subseteq \Ee_r\) for every \(r \in \N\). In \Cref{prop:E_r_H_r} we will see that this inclusion is strict.

%-------------------------------------------------------------
%-------------------------------------------------------------

\section{Exactness and excision theorems}
\label{sec:excision-exactness}

In this section we prove an exactness theorem and an excision theorem for the pages of the magnitude-path spectral sequence. Both statements involve a relative version of the reachability complex and they each hold only for certain nice subspace inclusions. In particular, the exactness theorem holds for {\(r\)-pairs} (\Cref{def:r-pair}), a notion that will feature later in our definition of \(r\)-cofibrations.

%-------------------------------------------------------------

\subsection{The exactness theorem}

Throughout this section we will be concerned with the \emph{relative magnitude-path spectral sequence} of a pair of \(\N\)-metric spaces. We can define the relative MPSS for any pair of a space \(X\) and a subspace \(A \subseteq X\). Because the inclusion \(i \colon A \hookrightarrow X\) is an isometry, the induced map \(\RC(i) \colon \RC(A) \to \RC(X)\) is a strict inclusion of filtered chain complexes. Thus, we can make the following definition.

\begin{defn}
\label{def:relative_MPSS_filtered}
    Let \(X\) be an object of \(\NMet\) and \(A \hookrightarrow X\) a subspace inclusion. The \emph{relative reachability complex} of the pair \((X,A)\) is the quotient chain complex
    \[\RC(X,A) = \RC(X)/\RC(A),\]
    equipped with the filtration inherited from that on \(\RC(X)\), so that 
    \[F_\ell(\RC(X,A)) = F_\ell(\RC(X)) / F_\ell(\RC(A)).\]
    The associated spectral sequence is called the \emph{relative magnitude-path spectral sequence} and denoted by \((E^r(X,A),\Delta^r)_{r\geq 0}\).
\end{defn}

\begin{rem}
    In \cite[Definition 3.2]{HepworthRoff2024}, a relative magnitude-path spectral sequence is defined for any pair of a directed graph \(X\) and a subgraph \(A \subseteq X\). That definition coincides with ours in the case that the subgraph \(A\) is \emph{convex} in the sense of Leinster \cite[Definition 4.1]{LeinsterGraph}.
\end{rem}

We are going to prove that each page of the magnitude-path spectral sequence has a property analogous to Eilenberg--Steenrod's exactness axiom for ordinary homology of topological spaces. Although the relative MPSS can be defined for any pair \((X,A)\), our exactness theorem will hold only for certain `good' pairs. Before we can state the theorem, we must characterize which pairs are `good'. For that, we need the following notion.

\begin{defn}
\label{def:in-out-closed}
    A subspace \(A\) is \demph{inward-closed in \(X\)} if it satisfies
    \[\dis_X(x,a) = \infty \text{ for all } x \in X \setminus A \text{ and } a \in A\]
    or equivalently
    \(\text{for all } a\in A \text{ and } x\in X,\;\dis_X(x,a)<\infty\Rightarrow x\in A.\)
    
     A subspace \(A\) is \demph{outward-closed in \(X\)} if it satisfies
    \[\dis_X(a,x) = \infty \text{ for all } a \in A \text{ and } x \in X \setminus A\]
    or equivalently
    \(\text{for all } a\in A \text{ and } x\in X, \; \dis_X(a,x)<\infty\Rightarrow x\in A.\)
    
\end{defn}

\begin{defn}
    \label{def:r-pair}
    Fix \(r \in \N\). An \demph{\(r\)-pair} of \(\N\)-metric spaces is a pair of \(\N\)-metric spaces \((X,U)\) such that \(U\) is a subspace of \(X\) and the following conditions hold.
    \begin{enumerate}
        \item \label{cond:r-pair_1}
        \(U\) is outward-closed in \(X\).% for all \(a \in A\) and \(x \in X \setminus A\) we have \(\dis(a,x) = \infty\).
        \item \label{cond:r-pair_2}
        For all  \(x \in X \setminus U\) and \(u \in U\) we have \(\dis(x,u) \geq r\).
    \end{enumerate}
    A map of \(r\)-pairs \((X,U) \to (Y,W)\) is a pair of maps \(f \colon X \to Y\) and \(g \colon U \to W\), such that \(f \circ i_U = i_W \circ g\), where \(i_U \colon U \hookrightarrow X\) and \(i_W \colon W \hookrightarrow Y\) are the inclusions.
\end{defn}

Note that every \(r\)-pair is an \((r-1)\)-pair, but not necessarily an \((r+1)\)-pair. Note also that when \(r \in\{0, 1\}\), condition \eqref{cond:r-pair_2} in \Cref{def:r-pair} is automatically satisfied.

The next theorem is one of the key technical results of this paper. It is a generalization of \cite[Theorem 3.3]{HepworthRoff2024}, which addresses the case \(r=1\) for directed graphs.

\begin{thm}[Exactness Theorem]
\label{thm:exactness}
    Fix \(r \in\N\), and let \((X,U)\) be an \(r\)-pair of \(\N\)-metric spaces. Let \(i \colon U \hookrightarrow X\) be the inclusion. For each \(0 \leq s \leq r-1\) we have a split short exact sequence of \(s\)-bigraded complexes on page \(s\) of the MPSS:
    \begin{equation}\label{eq:exactness1}
        0 \to E^s(U) \xrightarrow{E^s(i)} E^s(X)  \xrightarrow{} E^s(X,U) \to 0.
    \end{equation}
    On page \(r\), we have a short exact sequence of \(r\)-bigraded complexes, which splits as an exact sequence of underlying
    bigraded modules:
    \begin{equation}\label{eq:exactness2}
        0 \to E^r(U) \xrightarrow{E^r(i)} E^r(X)  \xrightarrow{} E^r(X,U) \to 0.
    \end{equation}
    On page \({r+1}\) we have the long exact sequence
    \begin{equation}\label{eq:exactness3}
        \cdots \to E^{r+1}_{p,q}(U) \xrightarrow{E^{r+1}(i)} E^{r+1}_{p,q}(X) \xrightarrow{} E^{r+1}_{p,q}(X, U) \xrightarrow{\partial_\ast} E^{r+1}_{p-r,q+r-1}(U) \to \cdots.
    \end{equation}
\end{thm}

\begin{proof}
    We first prove that the map \(E^s(i)\) is split injective, as a map of \(s\)-bigraded complexes, for \(0\leq s\leq r-1\), and as a map of bigraded modules for \(s=r\). In order to do this, we provide a retraction \(\pi:\RC(X)\to\RC(U)\) satisfying the conditions of \Cref{P:partial_filtered}.
    
    For each \(K \in \N\), let \(\RC^K(X)\) denote the submodule of \(\RC(X)\) spanned by those generators \((x_0,\ldots,x_{n})\) such that the size of the set \(\{k \mid x_k \in X \setminus U\}\) is \(K\). Then \(\RC^0(X)=\RC(U)\) and, because \(U\) is outward-closed, we have for \(K \geq 1\) that
    \[
    (x_0,\ldots,x_{n}) \in \RC^K(X) \Longleftrightarrow
    \begin{cases}
    x_k \in X \setminus U &  0 \leq k \leq K-1,\\
    x_k \in U  & K \leq k \leq n.
    \end{cases}
    \]
    Equip \(\RC^K(X)\) with the grading by \(n\) and the length filtration inherited from \(\RC(X)\). As a filtered graded module, \(\RC(X)\) splits into the direct sum
    \[\RC(X) = \displaystyle{\RC(U) \oplus \bigoplus_{K \geq 1} \RC^K(X)}.\]
    Let \(\pi \colon \RC(X) \to \RC(U)\) be the projection at the level of filtered graded modules.
    
    We claim that \(\pi\) satisfies the hypothesis of \Cref{P:partial_filtered}, namely that for every \(p\geq 0\) and for every generator \(\tuple{x}\) of \(\RC(X)\) of length \(p\), we have
    \[(\partial \pi-\pi \partial)(\tuple{x})\in F_{p-r}\RC(X).\]
    Indeed, if \(\tuple{x} \in \RC(U)\), or \(\tuple{x} \in \RC^K(X)\) with \(K \geq 2\), then \((\partial \pi-\pi \partial)(\tuple{x})=0\). Suppose \(\tuple{x} \in \RC^1(X)\). In that case, \(\tuple{x}\) has the form \((x_0,\ldots,x_{n})\) with \(x_0 \in X \setminus U\) and \(x_j \in U\) for \(j \geq 1\) and \((\partial \pi-\pi \partial)(\tuple{x})=-(x_1,\ldots,x_n)\). Since \(\dis_X(x_0,x_1)\geq r\), the length of \((x_1,\ldots,x_n)\) is \(p-\dis_X(x_0,x_1)\leq p-r\), as claimed.
    
    It now follows from \Cref{P:partial_filtered} that $\pi$ induces a map of \(s\)-bigraded complexes $E^s(\pi)\colon E^s(X)\to E^s(U)$ for \(0 \leq s < r\), and a map of bigraded modules \(E^s(\pi) \colon E^s(X) \to E^s(U)\). Clearly, $E^s(\pi)\circ E^s(i)=\id_{E^s(U)}$ for $0\leq s\leq r$. 

    We prove the first statement of the proposition by induction on \(s\). For \(s=0\), since the morphism of filtered complex \(\RC(U)\to\RC(X)\) is strict and injective,  by \cite[Proposition 1.1.11]{Deligne}, we have the short exact sequence
    \eqref{eq:exactness1} at page \(0\). We have proved that \(E^0(i)\) splits. Now take \(0<s<r-1\) and assume the statement holds for \(s-1\). The short exact sequence \eqref{eq:exactness1} of \((s-1)\)-bigraded complexes at page \(s-1\)  yields a long exact sequence at page \(s\) and since \(E^s(i)\) is split injective we get a split short exact sequence \eqref{eq:exactness1} of \(s\)-bigraded complexes. Combined with the induction hypothesis, that proves the first part of the statement. From the short exact sequence on page \(r-1\) we obtain a long exact sequence on page \(r\) which is split at the level of bigraded modules by the map \(E^r(\pi)\), yielding the short exact sequence in line \eqref{eq:exactness2}. That short exact sequence, in turn, gives us the long exact sequence on page $r+1$.
\end{proof}

%-------------------------------------------------------------

\subsection{The excision theorem}

We give two versions of an excision theorem, both of which will be used later. Each one specifies criteria for the pushout of a map along a subspace inclusion to induce an isomorphism of filtered relative reachability complexes, and thus an isomorphism on all pages of the relative magnitude-path spectral sequence. The proofs require the following lemma.

\begin{lem}
\label{lem:pushout_closed}
    Let \(X\) be an \(\N\)-metric space and \(i \colon A \hookrightarrow X\) a subspace inclusion. Let \(f\colon A \to B\) be any map in \(\NMet\) and consider the pushout
        \begin{equation}\label{diag:pushout}
        \begin{tikzcd}
            A \arrow[d, "f"'] \arrow[hook, r,"i"] \arrow[dr, phantom, "\scalebox{1.5}{\rotatebox{180}{$\lrcorner$}}" , very near end]
            & X \arrow{d}\arrow[d,"g"] \\
            B \arrow[hook, r,"j"'] & Y.
        \end{tikzcd}
        \end{equation}
    The map \(j \colon B \to Y\) is a subspace inclusion and \(g\) restricts to an isomorphism 
    \[g|_{X \setminus A} \colon X \setminus A \to Y \setminus B.\]
    Moreover, if \(A\) is inward-closed (or outward-closed) in \(X\) then so is \(B\) in \(Y\), and if \((X,A)\) is an \(r\)-pair then so is \((Y,B)\).
\end{lem}

\begin{proof}
    The metric on the pushout is described explicitly in \Cref{lem:NMet_pushout}\!~\eqref{pt:NMet_pushout_1}. It can be seen directly from that description that \(j\) is a subspace inclusion and that \(g\) restricts to an isomorphism (bijective isometry) between \(X \setminus A\) and \(Y \setminus B\).

    If we assume that \(A\) is outward-closed in \(X\), the same lemma tells us that for each \(y \in B\) and \(y' \in Y \setminus B\) we have \(y' = g(x)\) for a unique \(x \in X \setminus A\) and 
    \[\dis_Y(y,y') = \min_{a \in A} \{\dis_B(y,f(a)) + \dis_X(a, x)\} = \infty,\]
    since \(\dis_X(a,x) = \infty\) for all \(a \in A\). This says that \(B\) is outward-closed in \(Y\); the same argument shows that if \(A\) is inward-closed in \(X\), so is \(B\) in \(Y\). If \((X,A)\) is an \(r\)-pair, then for the same \(y\) and \(y'\) we have
    \[\dis_Y(y',y) = \min_{a \in A} \{\dis_X(x,a) + \dis_B(f(a), y)\} \geq r,\]
    since \(\dis_X(x,a) \geq r\) for all \(a \in A\). This says that \((Y,B)\) is an \(r\)-pair.
\end{proof}

\begin{notation}
    Given a metric space \(X\) and a point \(x \in X\), let 
    \[X_{\geq x} = \{x' \in X \mid \dis_X(x,x') < \infty\},\] 
   % and let
    \[X_{\leq x} = \{x' \in X \mid \dis_X(x',x) < \infty\},\]
    both metrized as subspaces of \(X\).
\end{notation}

\begin{thm}[General Excision Theorem]
\label{thm:general-excision}
    Let \(X\) be an \(\N\)-metric space and \(i \colon A \hookrightarrow X\) the inclusion of an outward-closed subspace. Let \(f \colon A \to B\) be any map in \(\NMet\) and consider their pushout, as shown in diagram \eqref{diag:pushout}. Suppose that for every \(x \in X \setminus A\) the map \(g\) restricts to a bijective isometry \(X_{\geq x} \to Y_{\geq g(x)}\). Then the map \(\bar{g} \colon \RC(X,A)\to \RC(Y,B)\) determined by \(g\) is an isomorphism of filtered chain complexes, so the induced map
    \[E^r(\bar{g}) \colon E^r(X, A) \to E^r(Y, B)\]
    is an isomorphism for every \(r \in \N\).
    
    If \(A\) is inward-closed in \(X\), the same result holds under the condition that for every \(x \in X \setminus A\) the map \(g\) restricts to a bijective isometry \(X_{\leq x} \to Y_{\leq g(x)}\).
\end{thm}

\begin{proof}
    Suppose \(A\) is outward-closed in \(X\); then, by \Cref{lem:pushout_closed}, \(j \colon B \to Y\) is the inclusion of an outward-closed subspace. By definition, the chain complex
    \[F_\ell(\RC(Y,B)) = F_\ell(\RC(Y))/F_\ell(\RC(B))\] 
    is freely generated in degree \(n\) by tuples
    \((y_0,\ldots, y_{n})\)
    of points in \(Y\) such that \(\sum_{i=0}^{n-1} \dis_Y(y_i,y_{i+1}) \leq \ell\) 
    % \Sarah{\(\sum_{i=0}^{n-1} \dis_Y(y_i,y_{i+1}) \leq \ell\) and related changes below }
    and \(y_k \in Y \setminus B\) for some \(k \in \{0, \ldots, n\}\).  Notice, though, that as \(B\) is outward-closed in \(Y\) we have \(y_k \in Y \setminus B\) for some \(k\) if and only if \(y_0 \in Y \setminus B\). Moreover, since
    \[\dis_Y(y_0, y_k) \leq \sum_{i=1}^{k} \dis_Y(y_{i-1},y_{i}) \leq \sum_{i=1}^{n} \dis_Y(y_{i-1},y_{i}) \leq \ell < \infty\]
    for \(1 \leq k \leq n\), the points \(y_1,\ldots, y_{n}\) all belong to the subspace \(Y_{\geq{y_0}}\). Thus, \(F_\ell(\RC_{n}(Y,B))\) is freely generated by the set
    \[\left\{(y_0,\ldots, y_{n}) \mid y_0 \in Y \setminus B, y_k \in Y_{\geq y_0} \text{ for } 1 \leq k \leq n, \text{ and } \sum_{i=1}^{n} \dis_Y(y_{i-1},y_{i}) \leq \ell\right\}.\]
    
    Fix a generator \((y_0,\ldots, y_{n})\). Since \(y_0\) is in \(Y \setminus B\), there exists a unique point \(x_0\in X \setminus A\) such that \(g(x_0)=y_0\). By assumption, the induced map \(X_{\geq x_0} \to Y_{\geq y_0}\) is a bijective isometry, so there are unique points \(x_1,\ldots, x_{n}\) in \(X\) such that  \(g(x_k)=y_k\), and they satisfy \(\sum_{i=1}^{n} \dis_X(x_{i-1},x_{i})\leq\ell\). It follows that the chain map \(\bar{g} \colon \RC(X,A)\to \RC(Y,B)\) is an isomorphism of filtered chain complexes and hence induces an isomorphism \(E^r(\bar{g}) \colon E^r(X, A) \to E^r(Y, B)\) for every \(r \in \N\).

    The proof is similar in the case that \(A\) (and hence \(B\)) is inward-closed, using the observation that \(y_k \in Y \setminus B\) for some \(k \in \{0, \ldots, n\}\) if and only if \(y_{n} \in Y \setminus B\).
\end{proof}

\begin{thm}[Special Excision Theorem]
\label{thm:special-excision}
    Let \(A \hookrightarrow U \hookrightarrow X\) be subspace inclusions such that \(A\) is inward-closed in \(X\) and \(U\) is outward-closed in \(X\). Let \(f \colon A\to B\) be any map in \(\NMet\), and consider the diagram
    \begin{equation}
        \label{diag:AUX_pushout}
        \begin{tikzcd}
        A \arrow[swap]{d}{f} \arrow[hook,swap]{r}{}  
        & U \arrow[]{d}{} \arrow[hook,swap]{r}{}  
        & X \arrow[]{d}{g} \\
        B  \arrow[swap, hook]{r}{} 
        & W \arrow[hook]{r}{}&  Y
        \end{tikzcd}
    \end{equation} 
    in which the left square and the outer rectangle are pushouts. The morphism \(\bar{g} \colon \RC(X, U) \to \RC(Y, W)\) determined by \(g\) is an isomorphism of filtered chain complexes, so the induced map
    \[E^r(\bar{g}) \colon E^r(X, U) \to E^r(Y, W)\]
    is an isomorphism for every \(r \in \N\).
\end{thm}

\begin{proof} 
    The pasting law for pushouts tells us that the diagram commutes and the square on the right is a pushout. We will apply the general excision theorem (\Cref{thm:general-excision}) to the square on the right.
    
    Fix \(x \in X \setminus U\). Because \(A\) is inward-closed in \(X\), we have \(\dis(x,a) = \infty\) for all \(a \in A\); it follows that \(X_{\geq x}\) is contained in \(X \setminus A\). By \Cref{lem:pushout_closed}, the subspace \(B\) is inward-closed in \(Y\), and as \(g(x) \in Y \setminus B\) we also have \(Y_{\geq g(x)} \subseteq Y\setminus B\). But \Cref{lem:pushout_closed} also tells us that \(g\) restricts to a bijective isometry \(X\setminus A\to Y\setminus B\), so it restricts to a bijective isometry \(X_{\geq x}\to Y_{\geq g(x)}\). Thus, the conditions of \Cref{thm:general-excision} are satisfied and this theorem follows.
\end{proof}

%-------------------------------------------------------------
%-------------------------------------------------------------

\section{A Mayer--Vietoris theorem}
\label{sec:cofibrations-and-MV}

In this section, we combine the exactness theorem and one of the excision theorems of the previous section to establish a {Mayer--Vietoris-type theorem} for each page of the magnitude-path spectral sequence. For each natural number \(r\), we introduce a class of maps that we call {\(r\)-cofibrations}, and the theorem takes the form of a long exact sequence on page \(E^{r+1}\) of the MPSS associated to an \(r\)-cofibration.

%-------------------------------------------------------------

\subsection{Cofibrations}\label{subsec:cofibrations}

For each natural number \(r\), we define a notion of \(r\)-cofibration in  \(\NMet\). These are particularly nice subspace inclusions, having a factorization in terms of a strong \(r\)-deformation retract (Definition~\ref{def:strong-r-def-retract}) and an \(r\)-pair (Definition~\ref{def:r-pair}).

\begin{defn}
\label{def:r-cofib}
    Fix \(r \in \N\). An \demph{\(r\)-cofibration} in \(\NMet\) is a subspace inclusion \(i \colon A \hookrightarrow X\) for which the following conditions hold.
    \begin{enumerate}
        \item \label{cond:r-cofib_1} The subspace \(A\) is inward-closed in \(X\).
        \item \label{cond:r-cofib_2} There exists a factorization \(A \hookrightarrow U \hookrightarrow X\) of \(i\) such that \((X,U)\) is an \(r\)-pair and \((U,A)\) is a strong \(r\)-deformation retract.
    \end{enumerate}
    For each \(r \in \N\), we will denote the class of \(r\)-cofibrations by \(\cof_r\).
\end{defn}

We have observed, on the one hand, that every \(r\)-pair is an \((r-1)\)-pair but not necessarily an \((r+1)\)-pair. On the other hand, every \(r\)-deformation retract is an \((r+1)\)-deformation retract but not necessarily an \((r-1)\)-deformation retract. It follows that the classes of \(r\)-cofibrations are not nested. However, the classes are not disjoint, as the next example shows.

\begin{eg}
\label{eg:special-cofibs}
    There are some maps in \(\NMet\) that belong to \(\cof_r\) for every \(r\).
    \begin{enumerate}
        \item \label{eg:empty-cofib} For every object \(X\) of \(\NMet\), the inclusion of the empty subspace \(\emptyset \hookrightarrow X\) is an \(r\)-cofibration for every \(r \in \N\).
        
        \item \label{eg:iso-cofib} If \(i \colon A \to Y\) is an isomorphism in \(\NMet\)---that is, a bijective isometry---then it is an \(r\)-cofibration for every \(r \in \N\). 
    \end{enumerate}
\end{eg}

A \(0\)-deformation retract is simply a bijective isometry. It follows that a \(0\)-cofibration is a subspace inclusion \(A \hookrightarrow X\) with the property that \(A\) is both inward-closed and outward-closed in \(X\). This says precisely that \(X = A \sqcup X \setminus A\), where \(\sqcup\) denotes the coproduct in \(\NMet\).

For \(r \geq 1\), an important class of examples of \(r\)-cofibrations arises via the \emph{\(r\)-cylinder construction}, which we now define.

\begin{defn}
\label{def:r-cylinder}
    Fix \(r \geq 1\). The \demph{directed \(r\)-interval} is \(K_r = (\{-2,-1,0,1,2\}, \dis)\) where \(\dis(-2,-1)=\dis(0,-1)=\dis(0,1)=\dis(2,1)=r\) and \(\dis(x,y)=+\infty\) otherwise. We can represent \(K_r\) graphically as
    \[
    \begin{tikzcd}[column sep=4ex]
        % -2 & -1 & 0 & 1 & 2 \\ [-20]
        -2 \arrow{r}{r} & -1 & 0 \arrow[swap]{l}{r} \arrow{r}{r} & 1 & 2 \arrow[swap]{l}{r} 
    \end{tikzcd}
    \]
    Given an \(\N\)-metric space \(X\), the \demph{\(r\)-cylinder} of \(X\) is the space 
    \[\Cylr(X) \coloneqq X \plusx K_r.\]
\end{defn}

We will prove that for every \(\N\)-metric space \(X\), the inclusion of the `end pieces' \(X \sqcup X \hookrightarrow \Cylr(X)\) is an \(r\)-cofibration. The proof uses the following lemma.

\begin{lem}
\label{lem:tensor-preserves-cof}
    Fix \(r \in \N\). For every \(\N\)-metric space \(X\), the functor
    \[X \plusx - \colon \NMet \to \NMet\]
    preserves \(r\)-pairs and the classes of morphisms \(\mathcal{H}_r\) and \(\cof_r.\)
\end{lem}

\begin{proof}
    These are straightforward checks, whose details we omit. If \((Y,U)\) is an \(r\)-pair, one uses the fact that \((X \plusx Y) \setminus (X \plusx U) = X \plusx (Y \setminus U)\) to check that \((X\plusx Y,X\plusx U)\) is an \(r\)-pair. If $f \colon Y \to Y'$ is in $\mathcal H_r$, it is immediate to see that $\id_X \plusx f$ is in $\mathcal H_r$. If \(i \colon A \hookrightarrow Y\) is an $r$-cofibration and \(A \hookrightarrow U \hookrightarrow Y\) is a factorization of \(i\) satisfying \Cref{def:r-cofib}, one can check that \(X \plusx A \hookrightarrow X \plusx U \hookrightarrow X \plusx Y\) also satisfies \Cref{def:r-cofib}, so \(\id_X \plusx i\) is an \(r\)-cofibration.
\end{proof}

\begin{prop}\label{prop:cylr-inclusions}
    Let \(X\) be an \(\N\)-metric space. For each \(r \geq 1\), the inclusion
    \[X \sqcup X \hookrightarrow \Cylr(X)\]
    that maps \(X \sqcup X\) isometrically to the subspace \(X \plusx \{-2,2\}\) is an \(r\)-cofibration.
\end{prop}

\begin{proof}
    First observe that the subspace \(\{-2,2\}\) of \(K_r\) is inward-closed. Let \(U\) be the subspace \(\{-2,-1,1,2\}\) of \(K_r\), which is isomorphic to \(I_r \sqcup I_r\). It is outward-closed, and \((K_r,U)\) is an \(r\)-pair; also, the map \(\{-2,2\} \hookrightarrow U\) is a strong \(r\)-deformation retract. Thus, \(\iota \colon \{-2,2\} \hookrightarrow K_r\) is in \(\cof_r\) and it follows from \Cref{lem:tensor-preserves-cof} that
    \[X \sqcup X \cong X \plusx \{-2,2\} \xhookrightarrow{\id_X \plusx \iota } X \plusx K_r = \Cylr(X)\]
    is in \(\cof_r\) as well.
\end{proof}

%-------------------------------------------------------------

\subsection{The Mayer--Vietoris theorem}

Using the excision theorem and the exactness theorem of \Cref{sec:excision-exactness}, we now prove that page \(E^{r+1}\) of the magnitude-path spectral sequence satisfies a Mayer--Vietoris-type theorem with respect to \(r\)-cofibrations.

\begin{thm}[A Mayer--Vietoris Sequence for the MPSS]
\label{thm:mayer-vietoris}
    Fix \(r \in \N\). Let \(i \colon A \hookrightarrow X\) be an \(r\)-cofibration, and \(f \colon A \to B\) be any map in \(\NMet\). Their pushout
    \begin{equation*}
        \begin{tikzcd}[column sep=3.5ex,row sep=3ex]
            A \arrow[d, "f"'] \arrow[hook, r,"i"] \arrow[dr, phantom, "\scalebox{1.5}{\rotatebox{180}{$\lrcorner$}}" , very near end]
            & X \arrow{d}\arrow[d,"g"] \\
            B \arrow[hook, r,"j"'] & Y
        \end{tikzcd}
        \end{equation*}
 gives rise to a long exact sequence on page \(E^{r+1}\) of the MPSS: 
    \begin{equation*}
    \cdots \to E^{r+1}_{p,q}(A) \xrightarrow{(i_\ast, -f_\ast)} E^{r+1}_{p,q}(X) \oplus E^{r+1}_{p,q}(B) \xrightarrow{g_\ast \oplus j_\ast} E^{r+1}_{p,q}(Y) \xrightarrow{\partial} E^{r+1}_{p-r,q+r-1}(A) \to \cdots.
    \end{equation*}
\end{thm}

\begin{proof}
    By unpacking the definition of \(r\)-cofibration, we may consider the following pushout diagram, with the same notation as in~\Cref{thm:special-excision}
    \begin{equation*}
        \begin{tikzcd}[column sep=3.5ex,row sep=3.5ex]
        A \arrow[swap]{d}{f} \arrow[hook,swap]{r}{}  \arrow[dr, phantom, "\scalebox{1.5}{\rotatebox{180}{$\lrcorner$}}" , very near end]
        & U \arrow[]{d}{h} \arrow[hook,swap]{r}{}  \arrow[dr, phantom, "\scalebox{1.5}{\rotatebox{180}{$\lrcorner$}}" , very near end]& X \arrow[]{d}{g} \\
        B  \arrow[swap, hook]{r}{} 
        & W \arrow[hook]{r}{}&  Y.
        \end{tikzcd}
    \end{equation*} 
    Since \((X,U)\) is an \(r\)-pair, \((Y, W)\) is an \(r\)-pair by \Cref{lem:pushout_closed}. From \Cref{thm:exactness} (Exactness), we obtain a diagram of long exact sequences
    \begin{equation*}
    \begin{tikzcd}[column sep=4.2ex]
        \cdots \arrow{r} & E_{p,q}^{r+1}(U) \arrow{r} \arrow{r}{E^{r+1}(h)} & E_{p,q}^{r+1}(X) \arrow{r} \arrow{d}{E^{r+1}(g)} & E_{p,q}^{r+1}(X,U) \arrow{r} \arrow{d}{E^{r+1}(\bar{g})} & E_{p-r,q+r-1}^{r+1}(U) \arrow{r} \arrow{d}{E^{r+1}(h)} & \cdots \\
        \cdots \arrow{r} & E_{p,q}^{r+1}(W) \arrow{r} & E_{p,q}^{r+1}(Y) \arrow{r} & E_{p,q}^{r+1}(Y,W) \arrow{r} & E_{p-r,q+r-1}^{r+1}(W) \arrow{r} & \cdots.
    \end{tikzcd}
    \end{equation*}
    As \(A\) is inward-closed and \(U\) is outward-closed in \(X\), \Cref{thm:special-excision} (Excision) says that every third map in this diagram, \(E^{r+1}(\bar{g})\), is an isomorphism. We can therefore apply the result of exercise 38 of \cite[p.159]{Hatcher} to obtain the long exact sequence
    \[
    \begin{tikzcd}[column sep=small]
        \cdots \arrow{r} & E_{p,q}^{r+1}(U) \arrow{r} & E_{p,q}^{r+1}(X) \oplus E_{p,q}^{r+1}(W) \arrow{r} & E_{p,q}^{r+1}(Y) \arrow{r} & E_{p-r,q+r-1}^{r+1}(U) \arrow{r} & \cdots.
    \end{tikzcd}
    \]
    Finally, since \((U,A)\) is a strong \(r\)-deformation retract, \((W,B)\) is a strong \(r\)-deformation retract by \Cref{prop:strong_def_pushout}. In particular, we have \(U \simeq_r A\) and \(W \simeq_r B\). By \Cref{prop:r-homotopy-invariance}, every \(r\)-homotopy equivalence induces an isomorphism on page \(r+1\) of the MPSS; thus, there are isomorphisms \(E^{r+1}(U) \cong E^{r+1}(A)\) and \(E^{r+1}(W) \cong E^{r+1}(B)\) that we can use to obtain the long exact sequence in the statement of the theorem.
\end{proof}

%-------------------------------------------------------------
%-------------------------------------------------------------

\section{Brown categories of \(\N\)-metric spaces}
\label{sec:brown-cats-on-nmet}

In this section we will establish a family of homotopy theories on the category \(\NMet\), related to the pages of the magnitude-path spectral sequence.

%-------------------------------------------------------------

\subsection{Brown's categories of cofibrant objects}
\label{subsec:brown-cats}

We begin by presenting the definition of a Brown category as our setting for homotopy theory. This definition is the dual of Brown's \emph{category of fibrant objects}~\cite[Section 1]{Brown1973} and is sometimes called a \emph{category of cofibrant objects}. 

\begin{defn}\label{def:brown-cat}
    A \emph{Brown category} is a category \(\cn{C}\) with all finite coproducts and equipped with the data of two classes of morphisms, \(\mathcal{W}\) (the \emph{weak equivalences}) and \(\mathcal{C}\) (the \emph{cofibrations}), satisfying the following axioms.
        \begin{enumerate}
            \item The class \(\mathcal{W}\) satisfies the two-out-of-three property and contains the isomorphisms. 
            \item The class \(\mathcal{C}\) is closed under composition and contains the isomorphisms. 
            \item Given a morphism \(i \in \mathcal{C}\) and any morphism \(f\) in \(\cn{C}\), the pushout of \(i\) along \(f\) exists and is in \(\mathcal{C}\). If \(i\) is in \(\mathcal{W}\cap \mathcal{C}\), so is the pushout along \(f\).
            \item For every object \(X\), there exists a factorization of the codiagonal \(X \sqcup  X \to X\) as the composite of a cofibration followed by a weak equivalence.
            \item For every object \(X\), the map
        \(\emptyset \to X\) is a cofibration.
        \end{enumerate}
    The class \(\mathcal{W}\cap \mathcal{C}\) is called the class of
    \emph{acyclic cofibrations}.
\end{defn}

\begin{rem}
\label{rem:factorization-lemma}
    Brown's Factorization Lemma~\cite[p.421]{Brown1973} says these axioms imply that \emph{any} map in the category \(\cn{C}\) can be factorized as a composite of a cofibration followed by a map that is left-inverse to an acyclic cofibration---in particular, a weak equivalence.
\end{rem}

Axiom~(4) prescribes, for each object \(X\) in \(\cn{C}\), the existence of some object \(\Cyl(X)\) through which the codiagonal map factors,
\[X \sqcup  X \to \Cyl(X)\to X,\]
such that the first map is a cofibration and the second is a weak equivalence. The object \(\Cyl(X)\) is not uniquely determined by this property. If it can be chosen in such a way that the assignment \(X \mapsto \Cyl(X)\) is functorial, then \((\cn{C}, \mathcal{W}, \mathcal{C})\) is said to be a \emph{Brown category with functorial cylinder object}.

Our main theorem, \Cref{thm:NMet-Brown-cat}, states that for each natural number \(r \geq 1\) the category \(\NMet\) carries the structure of a Brown category with functorial cylinder object, in which the class of weak equivalences is the class \(\Ee_r\) of \(r\)-quasi-isomorphisms and the class of cofibrations is the class \(\cof_r\) of \(r\)-cofibrations. The proof consists in checking that Axioms~(1)--(5) are satisfied. Of these, Axioms~(2) and (3) take the most work, which is the focus of the next section.

%-------------------------------------------------------------
\subsection{Stability of cofibrations}

In this section we prove stability properties of the class \(\cof_r\) of \(r\)-cofibrations in \(\NMet\). We first show stability under composition. Then we prove that the class \(\cof_r\) and the class \(\cof_r\cap \Ee_r\) of acyclic \(r\)-cofibrations are both stable under pushout.

\begin{prop}
\label{prop:cof_composition}
    For each \(r \in \N\), the class $\cof_r$ is stable under composition.
\end{prop}

\begin{proof}
    Suppose we are given \(r\)-cofibrations \(i \colon A \hookrightarrow X\) and \(j \colon X \hookrightarrow Y\), and a factorization of each of them satisfying \Cref{def:r-cofib}:
    \begin{equation}
    \label{eq:comp-cofib}
    \begin{tikzcd}
        A \arrow[hook]{rr}{i} \arrow[hook, shift left=1]{dr}{i_1} && X \arrow[hook]{rr}{j} \arrow[hook, shift left=1]{dr}{j_1} && Y \\
        & U \arrow[two heads, shift left=1]{ul}{\pi_U} \arrow[hook]{ur}{i_2} && V \arrow[two heads, shift left=1]{ul}{\pi_V} \arrow[hook]{ur}{j_2} 
    \end{tikzcd}
    \end{equation}
    Our aim is to prove that \(j \circ i \colon A \hookrightarrow Y\) is an \(r\)-cofibration. We first remark that since \(A\) is a subspace of  \(X\) and \(X\) is inward-closed in \(Y\), we have that \(A\) is inward-closed in \(Y\) and condition \eqref{cond:r-cofib_1} of \Cref{def:r-cofib} is satisfied.
    
    Let us consider diagram \eqref{eq:comp-cofib} and form the pullback of \(\pi_V\) along the subspace inclusion \(i_2\) 
    to obtain a span \((\pi_W,i_W)\colon U \leftarrow W \xhookrightarrow{} V\). By \Cref{lem:strong_def_pullback} there exists \(k\) such that \(k \colon U \rightleftarrows W : \pi_W\) is a strong \(r\)-deformation retract; we claim the factorization 
    \[A \xhookrightarrow{k \circ i_1} W \xhookrightarrow{j_2 \circ i_W} Y\]
    satisfies condition \eqref{cond:r-cofib_2} of \Cref{def:r-cofib}. By \Cref{lem:strong_def_compose}, we can compose the strong \(r\)-deformation retracts \(k \colon U \rightleftarrows W : \pi_W\) and \(i_1 \colon A \rightleftarrows U : \pi_U\) to see that  \(k\circ i_1 \colon A \rightleftarrows W : \pi_U\circ \pi_W\) is a strong \(r\)-deformation retract.

    Let us show that \((Y,W)\) is an \(r\)-pair. Recall that \(W = \{v \in V \mid \pi_V(v) \in U\}\). Suppose \(w \in W\) and \(y \in Y\) satisfy \(\dis(w,y) < \infty\). Since \(V\) is outward-closed in \(Y\) and \(w\) belongs to \(V\), the point \(y\) must be in \(V\). We can apply \(\pi_V\) to see that \(\dis(\pi_V(w), \pi_V(y)) \leq \dis(w,y) < \infty\), which tells us that \(\pi_V(y)\) must be in \(U\), since \(U\) is outward-closed in \(X\). It follows that \(y\) is in \(W\), and hence that \(W\) is outward-closed in \(Y\). It remains to show that for  \(y\in Y\setminus W\) and \(w\in W\) we have \(\dis(y,w)\geq r\). We study two cases, depending whether \(y\in Y\setminus V\) or \(y\in V\setminus W\). If \(y \in Y \setminus V\), then \(\dis(y,w) \geq r\) because \((Y,V)\) is an $r$-pair. If \(y \in V \setminus W\), then \(\pi_V(y) \in X \setminus U\) and \(\pi_V(w) \in U\), so \(\dis(\pi_V(y), \pi_V(w)) \geq r\) because \((X,U)\) is an $r$-pair. But \(\pi_V\) is a map in \(\NMet\), so we must have \(\dis(y,w) \geq \dis(\pi_V(y), \pi_V(w)) \geq r\) as claimed.
\end{proof}

\begin{prop}
\label{prop:cof_pushout}
    For each \(r \in \N\), the classes $\cof_r$ and $\cof_r \cap \Ee_r$  are stable under pushout along any map in \(\NMet\).
\end{prop}

\begin{proof}
    Let \(i \colon A \hookrightarrow X\) be an \(r\)-cofibration, and let \(f \colon A \to B\) be any map in \(\NMet\). Let \(j \colon B \to Y\) be the pushout of \(i\) along \(f\). By \Cref{lem:pushout_closed} the map \(j\) is a subspace inclusion. Choose a factorization \(A \xhookrightarrow{i_1} U \xhookrightarrow{i_2} X\) of \(i\) satisfying the conditions in \Cref{def:r-cofib}. By the pasting law, forming the pushout of \(i_1\) along \(f\) gives us a factorization \(B \hookrightarrow W \hookrightarrow Y\) of \(j\) fitting into a diagram
    \[
    \begin{tikzcd}
        A \arrow[swap]{d}{f} \arrow[hook]{r}{i_1}  &   U \arrow{d}{g|_U} \arrow[hook]{r}{i_2} & X \arrow{d}{g}  \\
       B \arrow[hook]{r}{j_1} & W \arrow[hook]{r}{j_2} & Y
    \end{tikzcd}
    \]
    in which both squares and the outer rectangle are pushouts.
    
    Since \(A\) is inward-closed in \(X\) and \((X,U)\) is an \(r\)-pair, \Cref{lem:pushout_closed} tells us \(B\) is inward-closed in \(Y\) and \((Y,W)\) is an \(r\)-pair. As \((U,A)\) is a strong \(r\)-deformation retract, \Cref{prop:strong_def_pushout} tells us that \((W,B)\) is a strong \(r\)-deformation retract. Thus, conditions \eqref{cond:r-cofib_1}  and \eqref{cond:r-cofib_2} of \Cref{def:r-cofib} are satisfied and \(j\) is an \(r\)-cofibration.

    Assume in addition that \(i\) is in \(\Ee_r\). Since \(i_1\in\mathcal H_r\), it is an \(r\)-quasi-isomorphism and, by the two-out-of-three property, so is \(i_2\). That is, \(E^{r+1}(i_2)\) is an isomorphism.  \Cref{thm:exactness} (the Exactness Theorem) gives us the long exact sequence
    \[\cdots \to E^{r+1}_{p,q}(U) \xrightarrow{E^{r+1}(i_2)} E^{r+1}_{p,q}(X) \xrightarrow{} E^{r+1}_{p,q}(X, U) \xrightarrow{\partial_\ast} E^{r+1}_{p-r,q+r-1}(U) \to \cdots\]
    and \(E^{r+1}(X,U) = 0\). By  \Cref{thm:special-excision} (the Special Excision Theorem) we deduce that \(E^{r+1}(Y,W) = 0\) and, because \((Y,W)\) is an \(r\)-pair, we can now apply \Cref{thm:exactness} in the case of \((Y,W)\) to see that the map \(E^{r+1}(j_2)\) is an isomorphism. This implies that \(j\in\Ee_r\).
\end{proof}

%-------------------------------------------------------------
\subsection{The main theorem}

We now collect our previous results to establish a family of Brown category structures on \(\NMet\).

\begin{thm}
\label{thm:NMet-Brown-cat}
    For each natural number \(r \geq 1\), \((\NMet, \Ee_r, \cof_r)\) is a Brown category with functorial cylinder object.
\end{thm}

\begin{proof}
    Axiom~(1) is straightforward, using \Cref{prop:r_htpy_equiv_compose} and the functoriality of the spectral sequence. We have seen in \Cref{eg:special-cofibs}~\eqref{eg:empty-cofib} that for each \(r \in \N\) the class \(\cof_r\) contains every initial map \(\emptyset \to X\), so Axiom (5) is satisfied. We have seen in \Cref{eg:special-cofibs}~\eqref{eg:iso-cofib} that \(\cof_r\) contains all isomorphisms, and in \Cref{prop:cof_composition} that it is closed under composition, so Axiom (2) is satisfied. Axiom (3) is proved in Proposition \ref{prop:cof_pushout}.

    We claim that for \(r \geq 1\) the functor \(\Cylr = - \plusx K_r \colon \NMet \to \NMet\) satisfies Axiom (4): that for each \(\N\)-metric space \(X\) there is a factorization of the codiagonal map as \(X \sqcup X \to \Cylr(X) \to X\) where the first map belongs to \(\cof_r\) and the second to \(\Ee_r\). Let \(\iota_0 \colon \pt \hookrightarrow K_r\) be the inclusion of the point \(0\), and let  \(\terminal \colon K_r \to \pt\) be the terminal map; then \(\terminal \circ \iota_0 = \id_{\pt}\). Let \(f\colon K_r\to K_r\) be defined by \(f(2)=1\), \(f(-2)=-1\), and \(f(x)=x\) otherwise. Then \(f\) is a map in \(\NMet\) and we have
    \[\iota_0 \circ \terminal \rightsquigarrow_r f\leftsquigarrow_r \id_{K_r}. \]
    Thus, \(\terminal \colon K_r \to \pt\) belongs to \(\mathcal{H}_r\). It follows from \Cref{lem:tensor-preserves-cof} that
    \[\Cylr(X) = X\plusx K_r  \xrightarrow{\id_X \plusx \terminal} X \plusx \pt \cong X\]
    belongs to \(\mathcal{H}_r\) and hence to \(\Ee_r\). We have also seen, in \Cref{prop:cylr-inclusions}, that \(id_X \plusx \iota:X\sqcup X\to\Cylr(X)\) belongs to \(\cof_r\). Since the composite \((\id_X\plusx \iota)\circ (\id_X\plusx \terminal)\) is the codiagonal map, this proves the claim.
\end{proof}

%-------------------------------------------------------------
%-------------------------------------------------------------

\section{Homotopy pushouts, suspensions and spheres}
\label{sec:homotopy-pushouts}

For each natural number \(r\), we use homotopy pushouts in the Brown category \((\NMet, \Ee_r, \cof_r)\) to define a {suspension functor} \(\Sigma_r\). This leads to a bigraded family of {sphere objects} in the category of \(\N\)-metric spaces. We analyse the MPSS of the spheres, making use of a convenient small model for the homotopy pushout. A comparison of different models also yields an example showing that there are strictly more \(r\)-quasi-isomorphisms than \(r\)-homotopy equivalences.

%-------------------------------------------------------------

\subsection{Homotopy pushouts}
\label{subsec:homotopy-pushouts}

We begin by discussing homotopy colimits, and in particular homotopy pushouts, in a Brown category. We note how homotopy pushouts can be modelled by a (usual) pushout involving the cylinder object.

Given a Brown category \((\cn{C}, \mathcal{W}, \mathcal{C})\), let us denote by \(\Ho(\cn{C})\) its homotopy category---that is, the localization of \(\cn{C}\) at the class \(\mathcal{W}\) of weak equivalences. For any indexing category \(\cn{D}\), we can define homotopy colimits over \(\cn{D}\) in \(\cn{C}\), when they exist, via a left adjoint \(\hocolim_\cn{D}\) to the constant diagram functor \(\Ho(\cn{C})\to \Ho(\cn{C}^\cn{D})\). In the case that \(\cn{D}\) is a finite direct category this adjoint does exist, by the dual of results of Cisinski~\cite{Cisinski2010}, using a Reedy cofibration category structure on \(\cn{C}^\cn{D}\) as explained by Meier~\cite[\S3]{Meier16}. (See also~\cite[Chapter 9]{RB09}.)

The adjoint is constructed as follows. Given a functor \(F \colon \cn{D} \to \cn{C}\), take a Reedy cofibrant replacement \(F' \to F\) in \(\cn{C}^\cn{D}\). The colimit of \(F'\) exists and in the category \(\Ho(\cn{C})\) one has \(\hocolim_\cn{D} F \cong \colim_\cn{D} F'\).

In particular, one can compute the homotopy pushout of a span \((f,g):X \leftarrow A \rightarrow Y\) as the pushout of the span \((f',g'):X' \leftarrow A \rightarrow Y'\) where \(f=pf'\) and \(g=qg'\) have each been factorized as a cofibration followed by a weak equivalence.

Note that a Brown category is left proper---that is, weak equivalences are stable under pushout along cofibrations---by the dual of~\cite[\S4, Lemma 2]{Brown1973}. It follows from this that if one of the maps, say \(f\), is a cofibration then the (usual) pushout  is the homotopy pushout. This means that a general homotopy pushout can be computed by replacing only one of the maps \(f\) or \(g\) by a cofibration.

In the Brown category \((\NMet, \cof_r, \Ee_r)\), this gives the following result, where the \(r\)-cylinder \(\Cylr\) has been defined in \Cref{def:r-cylinder}.

\begin{prop}\label{prop:r-mapping-cylinder}
    Fix \(r \geq 1\). Given a span \(X \xleftarrow{f} A \xrightarrow{g} Y\) in the Brown category \((\NMet,\Ee_r,\cof_r)\), its homotopy pushout is the pushout of the span
    \begin{equation}\label{span1}
        \xymatrix{X\sqcup Y& A\sqcup A\ar[rr]^-{(\id_A,-2)\sqcup (\id_A,2)}\ar[l]_{f\sqcup g} && \Cylr(A)}.
    \end{equation}
\end{prop}

\begin{proof}
    Consider the following diagram, in which every square is a pushout square:
    \begin{equation*}
        \xymatrix@R=3ex@C=5ex{ & A\sqcup A\ar[rr]^-{(\id_A,-2)\sqcup (\id_A,2)}\ar[d]^{f\sqcup \id_A} && \Cylr(A)\ar[d]\\
        A\ar[r]^-{i_2}\ar[d]_{g} & X\sqcup A\ar[rr]^-{u}\ar[d]^{\id_X\sqcup g} && \Cylr(f)\ar[d]\\
       Y\ar[r]^-{i_2} & X\sqcup Y\ar[rr]&& M_r(f,g).
        }
    \end{equation*}
    In particular the span \eqref{span1} and  the span \((g,u\circ i_2) \colon Y\leftarrow A\rightarrow \Cylr(f)\) have the same pushout, denoted \(M_r(f,g)\). According to the Brown Factorization Lemma (see \Cref{rem:factorization-lemma}), \(u\circ i_2\) is a cofibrant replacement of \(f\), hence \(M_r(f,g)\) is a representative of the homotopy pushout of the span \((f,g) \colon X \leftarrow A \rightarrow Y\).
\end{proof}

\begin{defn}
    The pushout \(M_r(f,g)\) in \Cref{prop:r-mapping-cylinder} will be called the \demph{mapping \(r\)-cylinder} of the pair \((f,g)\).
\end{defn}

%-------------------------------------------------------------

\subsection{A `short' homotopy pushout}
\label{subsec:short-mapping-space}

To allow for more efficient computations---in particular, to aid the study of suspensions and spheres in \Cref{subsec:spheres}---we are going to build a `shorter' model for the homotopy pushout. 

\begin{defn}\label{def:short-r-cylinder}
    Fix \(r > 0\). Let \(K_r\) be the directed \(r\)-interval, as in \Cref{def:r-cylinder}. Let \(J_r\) denote the subspace of \(K_r\) on the set \(\{-1,0,1\}\), and let \(J_r^\vee\) denote the pushout of the span \(\pt \leftarrow J_r \hookrightarrow K_r\). That is, \(J_r^\vee=(\{-2,\smallbullet,2\}, \dis)\) with \(\dis(-2,\smallbullet)=\dis(2,\smallbullet)=r\) and \(\dis(x,y)=+\infty\) otherwise. We can represent \(J_r^\vee\) graphically as
    \[\begin{tikzcd}[column sep=4ex]
        -2 \arrow{r}{r} & \smallbullet & 2 \arrow[swap]{l}{r} 
    \end{tikzcd}.
    \]
    Given \((f,g):B \leftarrow A \rightarrow C\) in \(\NMet\),
    let \(M_r^{\short}(f,g)\) denote the pushout of the span
    \[\xymatrix{B\sqcup C&A\sqcup A\ar[rr]^-{(\id_A,-2)\sqcup (\id_A,2)}\ar[l]_{f\sqcup g} && A\plusx J_r^\vee}.\]
    Let \(\pi \colon M_r(f,g) \to M_r^\short(f,g)\) denote the map that projects \(A\plusx J_r\) onto \(A\plusx\{\smallbullet\}\).
\end{defn}

\begin{rem}
    The mapping \(r\)-cylinder \(M_r(\id_A, \id_A)\) is the \(r\)-cylinder \(\Cyl_r(A)\). However, \(M_r^\short(\id_A,\id_A)\) should not be thought of as a cylinder object, because it does not satisfy Axiom~(4) in \Cref{def:brown-cat}: the inclusion 
    \[A \sqcup A \cong (A \plusx \{-2\}) \sqcup (A \plusx \{2\}) \rightarrow M_r^\short(\id_A,\id_A)\]
    is not an \(r\)-cofibration.
\end{rem}

To prove that the map \(\pi \colon M_r(f,g) \to M_r^\short(f,g)\) belongs to \(\Ee_r\), we will use the next lemma. Part~\eqref{pt:pi-excision-1} of the lemma allows us to apply \Cref{thm:exactness} (Exactness) and part~\eqref{pt:pi-excision-2} allows us to apply \Cref{thm:general-excision} (General Excision).

\begin{lem}\label{lem:pi-excision}
    Let \(B \xleftarrow{f} A \xrightarrow{g} C\) be a span in \(\NMet\). 
    \begin{enumerate}
        \item \((M_r(f,g), A\plusx J_r)\) and \((M_r^\short(f,g),A)\) are both \(r\)-pairs. \label{pt:pi-excision-1}
        \item For every \(x \in M_r(f,g) \setminus A \plusx J_r\), the map \(\pi \colon M_r(f,g) \to M_r^\short(f,g)\) restricts to a bijective isometry \(M_r(f,g)_{\geq x} \to M_r^\short(f,g)_{\geq \pi(x)}\). \label{pt:pi-excision-2}
    \end{enumerate}
\end{lem}

\begin{proof}
    Using \Cref{lem:NMet_pushout}, we can describe \(M_r(f,g)\) and \(M_r^\short(f,g)\) explicitly. The metric space \(M_r(f,g)\) has underlying set
    \[(B \times \{-2\}) \sqcup (A \times \{-1,0,1\}) \sqcup (C \times \{2\}).\]
    Within each component of this disjoint union, the metric is the \(\ell_1\)-product metric described in \Cref{def:NMet-csmc}. Between the components, we have the distances
    \[\dis((b,-2),(a,-1)) = \dis_B(b,f(a))+r, \quad 
         \dis((c,2),(a,1)) = \dis_C(c,g(a))+r,\] 
    for \(a \in A\), \(b\in B\) and \(c\in C\) and \(\dis((x,i),(y,j)) = + \infty\) otherwise. Meanwhile, \(M_r^\short(f,g)\) has underlying set
    \((B\times\{-2\})\sqcup (A\times\{\smallbullet\}) \sqcup (C\times\{2\});\)
    between the components we have the distances
    \[\dis((b,-2),(a,\smallbullet)) = \dis_B(b,f(a))+r , \quad 
        \dis((c,2),(a,\smallbullet)) = \dis_C(c,g(a))+r ,\] 
    for \(a \in A\), \(b\in B\), \(c\in C\) and \(\dis((x,i),(y,j)) = + \infty\) otherwise.
    
    From these descriptions it is clear that \((M_r(f,g), A\plusx J_r)\) and \((M_r^\short(f,g),A)\) are both \(r\)-pairs and that \(\pi \colon M_r(f,g) \to M_r^\short(f,g)\) restricts to a bijective isometry 
    \[(B\times\{-2\}) \sqcup (A\times\{-1\}) \to  (B\times\{-2\}) \sqcup (A\times\{\smallbullet\}).\]
    Now take \(x \in M_r(f,g) \setminus A \plusx J_r\); that is, \(x \in (B\times\{-2\}) \sqcup (C\times\{2\})\). Without loss of generality we may assume that \(x=(b,-2)\) for some \(b\in B\). Then, by inspection, \(M_r(f,g)_{\geq x}\) is contained in \((B \times\{-2\}) \sqcup (A\times\{-1\})\) and
    \(M_r^\short(f,g)_{\geq \pi(x)}\) is contained in \((B\times\{-2\}) \sqcup (A\times\{\smallbullet\})\).
    It follows that \(\pi\) restricts to a bijective isometry \(M_r(f,g)_{\geq x}\to M_r^\short(f,g)_{\geq \pi(x)}\).
\end{proof}

\begin{prop}\label{prop:shorthocolim}
    Given any span \(B \xleftarrow{f} A \xrightarrow{g} C\) in \(\NMet\), the map 
    \[\pi \colon M_r(f,g) \to M_r^\short(f,g)\] 
    belongs to $\Ee_r$.
\end{prop}

\begin{proof}
    First, consider the diagram
    \begin{equation}\label{diag:short-hocolim}
    \xymatrix@R=2.5ex@C=2.5ex{ & A\plusx J_r\ar[r]^-p\ar[d] & A\ar[d]\\
    A\sqcup A\ar[r]\ar[d]_{f\sqcup g} & A\plusx K_r\ar[r]\ar[d] & A\plusx J_r^\vee\ar[d]\\
    B\sqcup C\ar[r] & M_r(f,g)\ar[r]^-\pi & M_r^\short(f,g)
    }
    \end{equation}
    The upper square here is a pushout because the functor \(A\plusx-\colon \NMet\to\NMet\) is a left adjoint (\Cref{thm:NMet-csmc}) and hence preserves the pushout that defines \(J_r^\vee\). The lower right square must be a pushout diagram because the left square is (by definition of \(M_r(f,g)\)) and the total bottom square is (by definition of \(M_r^\short(f,g)\)). It follows that the right total square is a pushout.
    
    By \Cref{lem:pi-excision}~\eqref{pt:pi-excision-1}, \((M_r(f,g), A\plusx J_r)\) and \((M_r^\short(f,g),A)\) are \(r\)-pairs. Applying \Cref{thm:exactness} (Exactness) to both yields a diagram of long exact sequences
    \begin{equation*}
    \begin{tikzcd}[column sep=small, row sep=3ex]
        \cdots \arrow{r} & E_{p,q}^{r+1}(A \plusx J_r) \arrow{r} \arrow{d}{E^{r+1}(p)} & E_{p,q}^{r+1}(M_r(f,g)) \arrow{r} \arrow{d}{E^{r+1}(\pi)} & E_{p,q}^{r+1}(M_r(f,g),A \plusx J_r) \arrow{r} \arrow{d}{E^{r+1}(\bar{\pi})} & \cdots \\
        \cdots \arrow{r} & E_{p,q}^{r+1}(A) \arrow{r} & E_{p,q}^{r+1}(M_r^\short(f,g)) \arrow{r} & E_{p,q}^{r+1}(M_r^\short(f,g),A) \arrow{r} & \cdots.
    \end{tikzcd}
    \end{equation*}
    The map \(p\) is in \(\mathcal{H}_r\), so \(E^{r+1}(p)\) is an isomorphism. Meanwhile, \Cref{lem:pi-excision}~\eqref{pt:pi-excision-2} tells us that the map \(\pi\) satisfies the conditions of \Cref{thm:general-excision} (General Excision). Applying that theorem to the right total square in \eqref{diag:short-hocolim}, we see that the map
    \[E^{r+1}(\bar{\pi}) \colon E^{r+1}(M_r(f,g),A\plusx J_r)\to E^{r+1}(M_r^\short(f,g),A)\]
    is an isomorphism. Thus, \(E^{r+1}(\pi)\) is an isomorphism, as claimed.
\end{proof}

The main application of \Cref{prop:shorthocolim} is to show that the `short mapping cylinder' can be used to compute homotopy pushouts.

\begin{cor}\label{cor:hocolimdef}
   Given any span \(B \xleftarrow{f} A \xrightarrow{g} C\) in \(\NMet\), the object \(M_r^\short(f,g)\) represents its homotopy pushout in \((\NMet,\Ee_r,\cof_r)\).
   \end{cor}

\begin{proof}
    Because \(M_r\) and \(M_r^\short\) have been defined using pushout diagrams, which are functorial, they upgrade to functors
    \[M_r,M_r^\short \colon \NMet^{\smallbullet\leftarrow\smallbullet\rightarrow\smallbullet} \rightrightarrows \NMet,\]
    together with the natural transformation \(\pi \colon M_r\Rightarrow M_r^\short.\) The functor \(M_r\) preserves (pointwise) \(r\)-quasi-isormorphisms and so does  \(M_r^\short\) by \Cref{prop:shorthocolim} and by the two-out-of-three property. As a consequence the induced functors on the homotopy categories are isomorphic, so both are adjoint to the constant functor.
\end{proof}

From \Cref{prop:shorthocolim} we can also derive an example showing that for each \(r \geq 1\) there are strictly more \(r\)-quasi-isomorphisms than \(r\)-homotopy equivalences. The map \(\pi\) in the following statement is illustrated in \Cref{fig:in_E_r_not_H_r}.

\begin{prop}\label{prop:E_r_H_r}
    Fix \(r \geq 1\). Consider the span \(\pt \xleftarrow{\terminal} \pt \xrightarrow{\terminal} \pt\). Let
    \[\pi \colon M_r(\terminal,\terminal) \to M_r^\short(\terminal,\terminal)\]
    be the projection map in \Cref{def:short-r-cylinder}. Then \(\pi\) is in \(\Ee_r\) but not in \(\mathcal{H}_r\).
\end{prop}

\begin{proof}
    We have already seen, in \Cref{prop:shorthocolim}, that the map \(\pi\) belongs to \(\Ee_r\). The claim is that \(\pi\) is not in \(\mathcal{H}_r\); we will prove it by contradiction.
    
    Set \(X=M_r(\terminal,\terminal)\) and \(Y=M_r^\short(\terminal,\terminal)\). We first prove  that if \(f \colon Y\to Y\) satisfies \(f\rightsquigarrow_r \id_Y\) or \(f\leftsquigarrow_r \id_Y\) then \(f=\id_Y\). If \(f\rightsquigarrow_r \id_Y\), then we have for all \(y\in Y\) that \(\dis_Y(f(y),y)\leq r\). This implies \(f(a)=a\) and \(f(d)=d\) and \(f(b)\in\{a,b\}\). Since \(f\) is \(1\)-Lipschitz we have \(\dis_Y(f(d),f(b))\leq \dis_Y(d,b)=r\), so \(f(b)\in\{b,c,d\}\). Thus \(f(b)=b\) and similarly \(f(c)=c\). The second case is proved in the same way.
    
    Now suppose there exists \(i \colon Y\to X\) such that \(\pi\circ i\) and \(i\circ\pi \) are \(r\)-homotopy equivalent to the identity. Then, by the previous paragraph, we have \(\pi\circ i=\id_Y\). This implies that \(i(a)=a\) and \(i(d)=d\) and \(i(b)\in\{b_1,b_2,b_3\}\). But it is clear that there is no such map in \(\NMet\).
\end{proof}

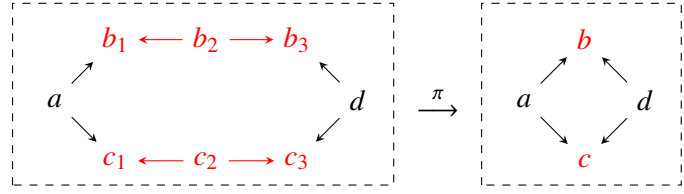
\begin{figure}
\centering
\begin{tabular}{p{5cm}p{0.5cm}p{3cm}}
\begin{tikzpicture}[scale=0.8,baseline=16]
      \begin{scope}
        \draw[dashed] (-0.7,-0.4) -- (-0.7,2.6) -- (5.6,2.6) -- (5.6,-0.4) -- (-0.7,-0.4);
            \node (a) at (0,1) {$a$};
              \node (b) at (1,2) {$\color{red}{b_1}$};
                \node (c) at (2.5,2) {$\color{red}b_2$};
                      \node (d) at (4,2) {$\color{red}b_3$};
                 \node (f) at (5,1) {$d$};
                    \node (g) at (1,0) {$\color{red}{c_1}$};
                       \node (h) at (2.5,0) {$\color{red}{c_2}$};
                          \node (i) at (4,0) {$\color{red}c_3$};
                             \draw[stealth-] (b)--(a);
                       \draw[color=red, stealth-] (b)--(c);
                        \draw[color=red, stealth-] (d)--(c);
                          \draw[stealth-] (d)--(f);
                            \draw[stealth-] (g)--(a);
                       \draw[color=red, stealth-] (g)--(h);
                        \draw[color=red, stealth-] (i)--(h);
                          \draw[stealth-] (i)--(f);
                   \end{scope}
         \end{tikzpicture}
         & $\overset{\pi}{\longrightarrow}$ &
    \begin{tikzpicture}[scale=0.8,baseline=16]
      \begin{scope}
        \draw[dashed] (-0.7,-0.4) -- (-0.7,2.6) -- (2.6,2.6) -- (2.6,-0.4) -- (-0.7,-0.4);
            \node (a) at (0,1) {$a$};
              \node (b) at (1,2) {$\color{red}{b}$};
                  \node (f) at (2,1) {$d$};
                    \node (g) at (1,0) {$\color{red}{c}$};
                        \draw[stealth-] (b)--(a);
                           \draw[stealth-] (b)--(f);
                            \draw[stealth-] (g)--(a);
                       \draw[stealth-] (g)--(f);
                                          \end{scope}
    \end{tikzpicture} 
    \end{tabular}
    \caption{The map \(\pi\) is in \(\Ee_r\) but not \(\mathcal{H}_r\); see \Cref{prop:E_r_H_r}. Arrows represent distances of \(r\). Where there is no arrow between distinct points, the distance is \(+\infty\). }
\label{fig:in_E_r_not_H_r}
\end{figure}

Combining \Cref{cor:hocolimdef} with the Mayer--Vietoris sequence in \Cref{thm:mayer-vietoris} yields a long exact sequence on every page of the MPPS, associated to any span in \(\NMet\). This will be useful in the next section, to analyse suspensions and spheres.

\begin{prop}
\label{prop:mayer_vietoris_hocolim}
    Given any span \(X \xleftarrow{f} A \xrightarrow{g} Y\) in \(\NMet\) we have, for each \(r \geq 1\), a long exact sequence on page \(E^{r+1}\) of the magnitude-path spectral sequence: 
    \[\cdots \to E^{r+1}_{p,q}(A) \rightarrow E^{r+1}_{p,q}(X) \oplus E^{r+1}_{p,q}(Y) \rightarrow E^{r+1}_{p,q}(M_r^\short(f,g)) \rightarrow E^{r+1}_{p-r,q+r-1}(A) \to \cdots.\]
\end{prop}

\begin{proof}
    Let \(W\) denote (any model for) the homotopy pushout in \((\NMet,\Ee_r,\cof_r)\) of the span \((f,g):X \leftarrow A \rightarrow Y\). Factorize \(g\) as \(g=qg'\) where \(g'\colon A\to Y'\) is in \(\cof_r\) and \(q\colon Y'\to Y\) is in \(\Ee_r\). As noted earlier, the homotopy pushout of the span \((f,g)\) can be computed as the pushout of the span  \((f,g'):X \leftarrow A \rightarrow Y'\). Applying the Mayer--Vietoris theorem to this span yields the long exact sequence
        \[\cdots \to E^{r+1}_{p,q}(A) \rightarrow E^{r+1}_{p,q}(X) \oplus E^{r+1}_{p,q}(Y') \rightarrow E^{r+1}_{p,q}(W) \rightarrow E^{r+1}_{p-r,q+r-1}(A) \to \cdots.\]
    The result follows, replacing \(Y'\) by the \(r\)-quasi-isomorphic space \(Y\) and \(W\) by the model of the homotopy pushout \(M_r^\short(f,g)\) given by~\Cref{cor:hocolimdef}.
\end{proof}

%-------------------------------------------------------------

\subsection{Suspensions and spheres}
\label{subsec:spheres}

In this section we define a {suspension} functor on \(\NMet\) and use it to construct an {\(r\)-sphere of dimension \(n\)} for each \(r \geq 1\) and \(n \in \N\). We apply the results of Sections \ref{subsec:homotopy-pushouts} and \ref{subsec:short-mapping-space} to analyse the MPSS of the spheres.

\begin{defn}\label{def:suspension-and-spheres}
    Given a non-empty \(\N\)-metric space \(X\), let \(\terminal_X \colon X\to\pt\) denote the terminal map. For each \(r \geq 1\), define the \demph{\(r\)-suspension} of \(X\) to be 
    \[\Sigma_rX \coloneqq M_r^\short(\terminal_X, \terminal_X).\]
    Since \(M_r^\short\) is defined by a pushout, this extends to a functor \(\Sigma_r \colon \NMet \to \NMet\). Let \(\mathbb S^0 \coloneqq \pt \sqcup \pt\). For \(r \geq 1\) and \(n \in \N\), define the \emph{\(r\)-sphere of dimension \(n\)} in \(\NMet\) to be
    \[\mathbb{S}_r^{n} \coloneqq \Sigma_r^n (\mathbb{S}^0).\]
    The spheres \(\mathbb{S}_r^{1}\) and \(\mathbb{S}_r^{2}\) are illustrated in \Cref{fig:r-spheres}. 
\end{defn}

In order to analyse the magnitude-path spectral sequence of the \(r\)-sphere of dimension \(n\), we first describe how the MPSS behaves with respect to \(r\)-suspension.

\begin{prop}\label{prop:suspension}
    Let \(X\) be a non-empty \(\N\)-metric space. For \(r \geq 1\) we have
    \[
    E_{p,q}^{r+1}(\Sigma_r X) \cong 
    \begin{cases}
        \kk &  (p,q) = (0,0) \\
        E_{p-r,q+r-1}^{r+1}(X) &  (p,q) \not\in \{(0,0), (r,1-r)\}
    \end{cases}
    \]
    while
    \[
    E_{r,1-r}^{r+1}(\Sigma_r X) \oplus \kk \cong E_{0,0}^{r+1}(X) \cong \kk\cdot(X\mod\sim_r).
    \]
\end{prop}

\begin{proof}
    Since \(\Sigma_rX\) is \(r\)-connected we have \(E^{r+1}_{0,0}(\Sigma_r X)\cong \kk\) by \Cref{prop:E^r_00}. \Cref{prop:mayer_vietoris_hocolim} gives us a long exact sequence on page \(r+1\) of the MPSS:
    \[
    \begin{tikzcd}[column sep=2ex]
    \cdots \: {}
        \arrow{r}
    & E^{r+1}_{p,q}(\pt) \oplus E^{r+1}_{p,q}(\pt)
        \arrow{r}
        \ar[draw=none]{d}[name=X, anchor=center]{}
    & E^{r+1}_{p,q}(\Sigma_rX)
        \ar[rounded corners,
            to path={ -- ([xshift=2ex]\tikztostart.east)
                      |- (X.center) \tikztonodes
                      -| ([xshift=-2ex]\tikztotarget.west)
                      -- (\tikztotarget)}]{dll}[at end]{} \\      [-10]
    E^{r+1}_{p-r,q+r-1}(X)
        \arrow{r}
    & E^{r+1}_{p-r,q+r-1}(\pt) \oplus E^{r+1}_{p-r,q+r-1}(\pt)
        \arrow{r}
    & {} \: \cdots.
    \end{tikzcd}
    \]
    By ~\Cref{prop:dimension} we know that \(E_{0,0}^{r+1}(\pt) = \kk\), while \(E_{p,q}^{r+1}(\pt) = 0\) for all \((p,q) \neq (0,0)\). It follows that, if \((p,q) \not\in \{(0,0), (r,1-r)\}\) then \(E_{p,q}^{r+1}(\Sigma_r X) \cong E_{p-r,q+r-1}^{r+1}(X)\). If \((p,q)=(r,1-r)\) then we have an exact sequence
    \[ 0 \to E_{r,1-r}^{r+1}(\Sigma_r X) \to E^{r+1}_{0,0}(X) \xrightarrow{\beta} E^{r+1}_{0,0}(\pt) \oplus E^{r+1}_{0,0}(\pt)\]
    so that \(E_{r,1-r}^{r+1}(\Sigma_r X)\) is isomorphic to the kernel of the map \(\beta\).
    The map \(\beta\) is the map \((E^{r+1}(\terminal_X), E^{r+1}(\terminal_X))\) induced by the two terminal maps that define the suspension. Its image is the span of \((1_\kk,1_\kk)\), and the induced short exact sequence
    \[ 0 \to E_{r,1-r}^{r+1}(\Sigma_r X) \to E^{r+1}_{0,0}(X) \to\kk\to 0\]
    splits. Thus \(E_{0,0}^{r+1}(X)\cong  E_{r,1-r}^{r+1}(\Sigma_r X) \oplus \kk\).
\end{proof}

Finally, we apply \Cref{prop:suspension} to describe page \(E^{r+1}\) for the \(r\)-sphere of dimension \(n\), for each \(n \in \N\). We find that \(E^{r+1}(\mathbb{S}_r^n)\) is concentrated on the line of slope \((1-r)/r\) through bigrading \((0,0)\), where it consists of a copy of the ordinary homology of the topological \(n\)-sphere. \Cref{fig:r-spheres} illustrates this for \(r,n \in \{1,2\}\).

\begin{thm}\label{thm:spheres}
    Fix \(r \geq 1\). We have
    \[
    E_{p,q}^{r+1}(\mathbb S_r^0) =
        \begin{cases}
        \kk\oplus \kk & \text {if } (p,q)=(0,0)\\
        0 & \text{otherwise},
        \end{cases}
    \]
    while for \(n \geq 1\) we have
    \[
    E_{p,q}^{r+1}(\mathbb S_r^n) =
        \begin{cases}
        \kk & \text{if } (p,q)=(0,0) \\
        \kk & \text{if } (p,q)=(nr,n(1-r))\\
        0 & \text{otherwise}.
        \end{cases}
    \]
\end{thm}

\begin{proof}
    We prove this by induction on \(n\). For the base case, observe that for every \(r \geq 1\) we have \(\mathbb{S}_r^0 = \Sigma_r^0(\mathbb{S}^0) = \pt \sqcup \pt\). Thus, by \Cref{prop:additivity},
    \[E_{p,q}^{r+1}(\mathbb{S}_r^0) \cong E_{p,q}^{r+1}(\pt) \oplus E_{p,q}^{r+1}(\pt), \]
    giving the first claim. Now take \(n \geq 1\) and suppose the statement holds for \(n-1\). Since \(\mathbb{S}_r^n = \Sigma_r(\mathbb{S}_r^{n-1})\) for \(n \geq 1\), \Cref{prop:suspension} tells us that \(E_{0,0}^{r+1}(\mathbb{S}_r^n) \cong \kk\) and that
    \[
    E_{r,1-r}^{r+1}(\mathbb{S}_r^n) \oplus \kk \cong E_{0,0}^{r+1}(\mathbb{S}_r^{n-1}) \cong 
    \begin{cases} 
    \kk \oplus \kk & n = 1 \\
    \kk & n > 1,
    \end{cases}
    \]
    so \(E_{r,1-r}^{r+1}(\mathbb{S}_r^1) \cong \kk\) and \(E_{r,1-r}^{r+1}(\mathbb{S}_r^n) = 0\) for all \(n > 1\). For \((p,q) \not\in \{(0,0), (r,1-r)\}\), \Cref{prop:suspension} combined with the induction hypothesis gives
    \[
    E_{p,q}^{r+1}(\mathbb{S}_r^n) \cong E_{p-r,q+r-1}^{r+1}(\mathbb{S}_r^{n-1}) \cong 
    \begin{cases}
        \kk & \text{if }(p-r,q+r-1) = ((n-1)r, (n-1)(1-r)) \\
        0 & \text{otherwise,}
    \end{cases}
    \]
    so that \(E_{p,q}^{r+1}(\mathbb{S}_r^n) \cong \kk\) if \((p,q) = (nr,n(1-r))\), and otherwise \(E_{p,q}^{r+1}(\mathbb{S}_r^n) = 0\).
\end{proof}

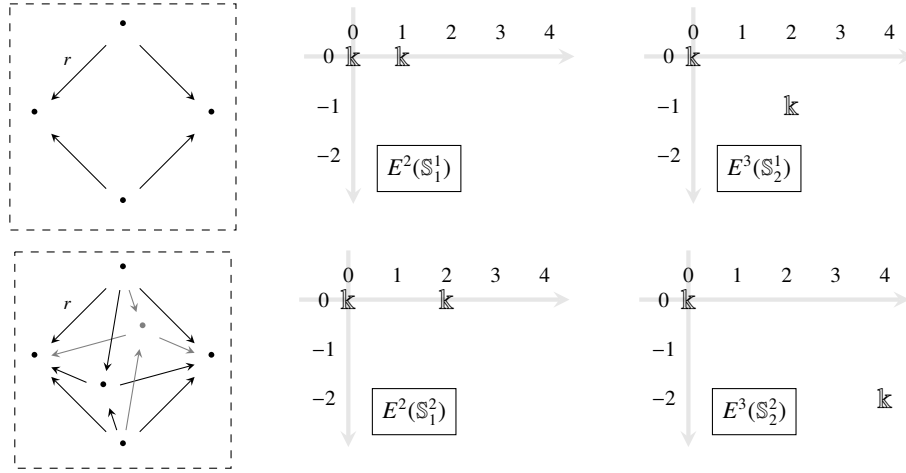
\begin{figure}
    \centering
    \begin{tikzpicture}[scale=.65,baseline=0]
        \draw[dashed] (-2.3,2.3) -- (2.3,2.3) -- (2.3,-2.3) -- (-2.3,-2.3) -- (-2.3,2.3);
        
        \node (a) at (0:1.8) {$\smallbullet$};
        \node (c) at (180:1.8) {$\smallbullet$};
        \node (e) at (90:1.8) {$\smallbullet$};
        \node (f) at (270:1.8) {$\smallbullet$};

        \draw[stealth-] (a) -- (e);
        \draw[stealth-] (c) -- (e);

        \draw[stealth-] (a) -- (f);
        \draw[stealth-] (c) -- (f);

        \node at (135:1.6) {\tiny{$r$}};
    \end{tikzpicture}
    \qquad
    \begin{tikzpicture}[baseline={(0,-.8)},xscale=.65,yscale=.65]
        \draw[-stealth, ultra thick,white!90!black] (-1,0) to (4.5,0);
        \draw[-stealth, ultra thick,white!90!black] (0,1) to (0,-3);
        \node (a) at (0,0) {$\kk$};
        \node (c) at (1,0) {$\kk$};
        \node() at (0,0.5) {$\scriptstyle 0$};
        \node() at (1,0.5) {$\scriptstyle 1$};
        \node() at (2,0.5) {$\scriptstyle 2$};
        \node() at (3,0.5) {$\scriptstyle 3$};
        \node() at (4,0.5) {$\scriptstyle 4$};
        \node() at (-0.5,0) {$\scriptstyle 0$};
        \node() at (-0.5,-1) {$\scriptstyle -1$};
        \node() at (-0.5,-2) {$\scriptstyle -2$};

        \node[draw, below left]() at (2.2,-1.8) {\footnotesize{$E^2(\mathbb{S}_1^1)$}};
    \end{tikzpicture}
    \qquad
    \begin{tikzpicture}[baseline={(0,-.8)},xscale=.65,yscale=.65]
        \draw[-stealth, ultra thick,white!90!black] (-1,0) to (4.5,0);
        \draw[-stealth, ultra thick,white!90!black] (0,1) to (0,-3);
        \node (a) at (0,0) {$\kk$};
        \node (c) at (2,-1) {$\kk$};
        \node() at (0,0.5) {$\scriptstyle 0$};
        \node() at (1,0.5) {$\scriptstyle 1$};
        \node() at (2,0.5) {$\scriptstyle 2$};
        \node() at (3,0.5) {$\scriptstyle 3$};
        \node() at (4,0.5) {$\scriptstyle 4$};
        \node() at (-0.5,0) {$\scriptstyle 0$};
        \node() at (-0.5,-1) {$\scriptstyle -1$};
        \node() at (-0.5,-2) {$\scriptstyle -2$};
        
        \node[draw, below left]() at (2.2,-1.8) {\footnotesize{$E^3(\mathbb{S}_2^1)$}};
    \end{tikzpicture}

    \vspace{2mm}
    
    \begin{tikzpicture}[scale=.65,baseline=0]
        \draw[dashed] (-2.2,2.2) -- (2.2,2.2) -- (2.2,-2.2) -- (-2.2,-2.2) -- (-2.2,2.2);
        
        \node (a) at (0:1.8) {$\smallbullet$};
        \node [gray] (b) at (.4,.6) {$\smallbullet$};
        \node (c) at (180:1.8) {$\smallbullet$};
        \node (d) at (-.4,-.6) {$\smallbullet$};
        \node (e) at (90:1.8) {$\smallbullet$};
        \node (f) at (270:1.8) {$\smallbullet$};

        \draw[gray, stealth-] (a) -- (b);
        \draw[gray, stealth-] (c) -- (b);
        \draw[stealth-] (a) -- (d);
        \draw[stealth-] (c) -- (d);
        
        \draw[stealth-] (a) -- (e);
        \draw[gray, stealth-] (b) -- (e);
        \draw[stealth-] (c) -- (e);
        \draw[stealth-] (d) -- (e);

        \draw[stealth-] (a) -- (f);
        \draw[gray, stealth-] (b) -- (f);
        \draw[stealth-] (c) -- (f);
        \draw[stealth-] (d) -- (f);

        \node at (135:1.6) {\tiny{$r$}};
    \end{tikzpicture}
    \qquad
    \begin{tikzpicture}[baseline={(0,-.8)},xscale=.65,yscale=.65]
        \draw[-stealth, ultra thick,white!90!black] (-1,0) to (4.5,0);
        \draw[-stealth, ultra thick,white!90!black] (0,1) to (0,-3);
        \node (a) at (0,0) {$\kk$};
        \node (c) at (2,0) {$\kk$};
        \node() at (0,0.5) {$\scriptstyle 0$};
        \node() at (1,0.5) {$\scriptstyle 1$};
        \node() at (2,0.5) {$\scriptstyle 2$};
        \node() at (3,0.5) {$\scriptstyle 3$};
        \node() at (4,0.5) {$\scriptstyle 4$};
        \node() at (-0.5,0) {$\scriptstyle 0$};
        \node() at (-0.5,-1) {$\scriptstyle -1$};
        \node() at (-0.5,-2) {$\scriptstyle -2$};
        
        \node[draw, below left]() at (2.2,-1.8) {\footnotesize{$E^2(\mathbb{S}_1^2)$}};
    \end{tikzpicture}
    \qquad
    \begin{tikzpicture}[baseline={(0,-.8)},xscale=.65,yscale=.65]
        \draw[-stealth, ultra thick,white!90!black] (-1,0) to (4.5,0);
        \draw[-stealth, ultra thick,white!90!black] (0,1) to (0,-3);
        \node (a) at (0,0) {$\kk$};
        \node (c) at (4,-2) {$\kk$};
        \node() at (0,0.5) {$\scriptstyle 0$};
        \node() at (1,0.5) {$\scriptstyle 1$};
        \node() at (2,0.5) {$\scriptstyle 2$};
        \node() at (3,0.5) {$\scriptstyle 3$};
        \node() at (4,0.5) {$\scriptstyle 4$};
        \node() at (-0.5,0) {$\scriptstyle 0$};
        \node() at (-0.5,-1) {$\scriptstyle -1$};
        \node() at (-0.5,-2) {$\scriptstyle -2$};

        \node[draw, below left]() at (2.2,-1.8) {\footnotesize{$E^3(\mathbb{S}_2^2)$}};
    \end{tikzpicture}
    \caption{The spheres \(\mathbb{S}_r^1\) and \(\mathbb{S}_r^2\), and page \(E^{r+1}\) of the MPSS for \(r=1,2\) (\Cref{thm:spheres}). Arrows represent distances of \(r\). Where there is no arrow between distinct points, the distance is \(+\infty\).}
    \label{fig:r-spheres}
\end{figure}

Note also that for every \(r \geq 1\) and \(n \in \N\) we have \(E^s(\mathbb{S}_r^n) = E^0(\mathbb{S}_r^n)\) for \(0\leq s\leq r\), while \(E^s(\mathbb{S}_r^n) = E^{r+1}(\mathbb{S}_r^n)\) for \(s \geq r +1\).

%-------------------------------------------------------------
%-------------------------------------------------------------

\section{The case of directed graphs}
\label{sec:brown-cat-on-digraph}

In this closing section, we restrict the focus to directed graphs. Specializing the definitions and results of \Cref{sec:cofibrations-and-MV} yields a class \(\cof_1\) of cofibrations in \(\DiGraph\) which satisfy a Mayer--Vietoris theorem for bigraded path homology and path homology. Specializing the results of \Cref{sec:brown-cats-on-nmet} yields a Brown category structure on \(\DiGraph\) in which the class of cofibrations is \(\cof_1\) and the weak equivalences are maps inducing isomorphisms on bigraded path homology.

%-------------------------------------------------------------

\subsection{Cofibrations of directed graphs}

We begin by describing explicitly the maps in \(\DiGraph\) that we will take as our cofibrations. First, let \(M \colon \DiGraph \hookrightarrow \NMet\) be the functor equipping each directed graph with its shortest path metric. Observe that if \(H \hookrightarrow K\) is the inclusion of an induced subgraph with the property that there are no edges from vertices in \(V(K) \setminus V(H)\) to those in \(V(H)\), then \(M(H) \hookrightarrow M(K)\) is the inclusion of a metric subspace. We will call \(H\) a \emph{strong deformation retract} of \(K\) if \(M(H)\) is a strong 1-deformation retract of \(M(K)\) in the sense of \Cref{def:strong-r-def-retract}.

\begin{defn}
\label{def:digraph-cofib}
    A \demph{cofibration of directed graphs} is the inclusion of an induced subgraph \(i \colon H \hookrightarrow G\) for which the following conditions hold.
    \begin{enumerate}
        \item \label{cond:digraph-cofib_1} There are no edges from vertices in \(V(G) \setminus V(H)\) to those in \(V(H)\).
        \item \label{cond:digraph-cofib_2} There exists a factorization \(H \hookrightarrow K \hookrightarrow G\) of \(i\) such that
        \begin{enumerate}
            \item There are no edges from vertices in \(V(K)\) to those in \(V(G) \setminus V(K)\).
            \item The subgraph \(H\) is a strong deformation retract of \(K\).
        \end{enumerate}
    \end{enumerate}
   An example can be seen on the left in \Cref{fig:cofib_not_1cofib}.
\end{defn}

The proof of the next lemma is a direct check.

\begin{lem}\label{lem:cof_DG}
    A map  \(i \colon H \to G\) in \(\DiGraph\) is a cofibration of directed graphs if and only if the map \(M(i) \colon M(H) \to M(G)\) in \(\NMet\) belongs to \(\cof_1\). \qed
\end{lem}

In a mild abuse of notation, we will also denote the class of cofibrations of directed graphs by \(\cof_1\).

\begin{rem}
    The class \(\cof_1\) in \(\DiGraph\) is obtained by pulling back the class of 1-cofibrations in \(\NMet\) along the functor \(M\). One can also pull back \(\cof_r\) for other values of \(r\), but the resulting classes of \(r\)-cofibrations of directed graphs are less interesting. Indeed, a subgraph inclusion \(i \colon H \hookrightarrow G\) is a 0-cofibration in this sense if and only if \(G\) decomposes as the coproduct \(H \sqcup G \setminus H\). For \(r > 1\) the map \(i\) is an \(r\)-cofibration if and only if \(G\) decomposes as a coproduct \(K \sqcup G \setminus K\) such that \(K\) contains \(H\) as a strong \(r\)-deformation retract.
\end{rem}

%-------------------------------------------------------------

\subsection{Two Mayer--Vietoris sequences}
\label{subsec:MV-Digraph}

Next, we describe how the Mayer--Vietoris sequence in \Cref{thm:mayer-vietoris} specializes to the category of directed graphs. That theorem says that if \(i \colon A \hookrightarrow X\) is an \(r\)-cofibration in \(\NMet\), then any pushout along \(i\) yields a long exact sequence on page \(E^{r+1}\) of the magnitude-path spectral sequence. Since cofibrations of directed graphs are 1-cofibrations, this specializes to a statement about page \(E^2\) of the MPSS. 

Recall from \Cref{subsec:MPSS} that the top horizontal axis of page \(E^2\) coincides with reduced path homology, denoted \(\GLMY\); the entirety of \(E^2\) is known as \emph{bigraded path homology} and denoted \(\PH\). Thus, the statements in this section concern bigraded path homology and reduced path homology. For both these functors, there already exist Mayer--Vietoris-type theorems: see \cite[Theorem~6.6]{HepworthWillerton2017}, \cite[Theorem~6.8]{HepworthRoff2024} and \cite[Theorem~3.25]{PH-eilenbergsteenrod}. Ours are novel in that they hold for a new class of cofibrations.

\begin{thm}[A Mayer--Vietoris Sequence for Bigraded Path Homology]
\label{thm:mayer-vietoris-BPH}
    Let \(i \colon H \hookrightarrow G\) be in \(\cof_1\), and let \(f \colon H \to K\) be any map in \(\DiGraph\). The pushout
    \begin{equation}
    \label{diag:digraph-pushout}
    \begin{tikzcd}
        H \arrow[hook, swap]{d}{i} \arrow{r}{f} \arrow[dr, phantom, "\scalebox{1.5}{\rotatebox{180}{$\lrcorner$}}" , very near end] & K \arrow{d}{j}\\
        G \arrow[swap]{r}{g} & G \cup_H K
    \end{tikzcd}
    \end{equation}
    gives rise to a long exact sequence on bigraded path homology:
    \[
    \begin{tikzcd}
    \cdots \: \PH_{p,q}(H)
        \arrow{r}{(i_\ast, -f_\ast)} 
    & \PH_{p,q}(G) \oplus \PH_{p,q}(K)
        \arrow{r}{g_\ast \oplus j_\ast}
        \ar[draw=none]{d}[name=X, anchor=center]{}
    & \PH_{p,q}(G \cup_H K)  
        \ar[rounded corners,
            to path={ -- ([xshift=2ex]\tikztostart.east)
                      |- (X.center) \tikztonodes
                      -| ([xshift=-2ex]\tikztotarget.west)
                      -- (\tikztotarget)}]{dll}[at end]{} \\      
    \PH_{p-1, q}(H)
        \arrow{r}{(i_\ast, -f_\ast)} 
    & \PH_{p-1, q}(G) \oplus \PH_{p-1, q}(K) 
        \arrow{r}{g_\ast \oplus j_\ast}
    & \PH_{p-1,q}(G \cup_H K) \: \cdots.
    \end{tikzcd}
    \]
\end{thm}

\begin{proof}
    Since \(M \colon \DiGraph \hookrightarrow \NMet\) is a left adjoint (\Cref{lem:coreflective}), applying \(M\) to \eqref{diag:digraph-pushout} produces a pushout diagram in \(\NMet\) in which, by \Cref{lem:cof_DG}, the left leg is a 1-cofibration. Applying \Cref{thm:mayer-vietoris} to that pushout diagram, and taking \(r=1\), we obtain a long exact sequence on page \(E^2\) of the MPSS---that is, on bigraded path homology.
\end{proof}

By considering only the top horizontal axis of \(E^2\), we can write down the new Mayer--Vietoris sequence for path homology.

\begin{cor}[A Mayer--Vietoris Sequence for Path Homology]
\label{cor:mayer-vietoris-GLMY}
    Under the conditions of \Cref{thm:mayer-vietoris-BPH}, there is a long exact sequence on reduced path homology:
    \[
    \begin{tikzcd}
    \cdots \: \GLMY_{n}(H)
        \arrow{r}{(i_\ast, -f_\ast)} 
    & \GLMY_{n}(G) \oplus \GLMY_{n}(K)
        \arrow{r}{g_\ast \oplus j_\ast}
        \ar[draw=none]{d}[name=X, anchor=center]{}
    & \GLMY_{n}(G \cup_H K)  
        \ar[rounded corners,
            to path={ -- ([xshift=2ex]\tikztostart.east)
                      |- (X.center) \tikztonodes
                      -| ([xshift=-2ex]\tikztotarget.west)
                      -- (\tikztotarget)}]{dll}[at end]{} \\      
    \GLMY_{n-1}(H)
        \arrow{r}{(i_\ast, -f_\ast)} 
    & \GLMY_{n-1}(G) \oplus \GLMY_{n-1}(K) 
        \arrow{r}{g_\ast \oplus j_\ast}
    & \GLMY_{n-1}(G \cup_H K) \: \cdots.
    \end{tikzcd}
    \]
\end{cor}

%-------------------------------------------------------------

\subsection{A Brown category of directed graphs}
\label{subsec:Brown-cat-Digraph}

We close by verifying that the Brown category structure on \(\NMet\) established in \Cref{thm:NMet-Brown-cat} restricts, when \(r=1\), to a Brown category structure on \(\DiGraph\). One can compare this to the structure established by Hepworth and Roff in \cite{HepworthRoff2024}, which has the same weak equivalences; we will demonstrate that our cofibrations are distinct from theirs.

In a mild abuse of notation, let us now denote by \(\Ee_1\) the class of morphisms in \(\DiGraph\) inducing an isomorphism on bigraded path homology.

\begin{thm}
\label{thm:DiGraph-Brown-cat}
    \((\DiGraph, \Ee_1, \cof_1)\) is a Brown category with functorial cylinder.
\end{thm}

\begin{proof}
    Axiom~(1) in \Cref{def:brown-cat} follows from the fact that \(\PH\) is a functor. Axiom~(2) follows from \Cref{prop:cof_composition} and Axiom~(3) from \Cref{prop:cof_pushout}, using \Cref{lem:cof_DG} and the fact that \(M\) creates pushouts. For every directed graph \(G\) the 1-cylinder \(\Cyl_1(M(G))\) is itself in the image of \(M\) and provides a factorization of the codiagonal which satisfies Axiom~(4). \Cref{eg:special-cofibs}\!~\eqref{eg:empty-cofib} tells us that Axiom~(5) is satisfied. 
\end{proof}

In \cite{HepworthRoff2024}, Hepworth and Roff give \(\DiGraph\) the structure of a Brown category in which the class of weak equivalences is \(\Ee_1\) and the cofibrations form a class we will denote by \(\cof_{\mathrm{HR}}\). The maps in \(\cof_{\mathrm{HR}}\) were introduced by Carranza \textit{et al} in \cite{CDKOSW} in a transposed form; we recall the definition as it is given in \cite[Definition 6.2]{HepworthRoff2024}. 

\begin{defn}
    A subgraph inclusion \(H \hookrightarrow G\) belongs to \(\cof_{\mathrm{HR}}\) if $H$ is inward-closed in  \(G\) and satisfies the following condition. For each \(x\in rH\), where \(rH\) denotes the maximal outward-closed subgraph of \(G\) containing \(H\), there is a vertex \(\rho(x)\in H\) with the property that \(\dis(h,x)=\dis(h,\rho(x))+\dis(\rho(x),x)\) for every \(h\in H\).
\end{defn}

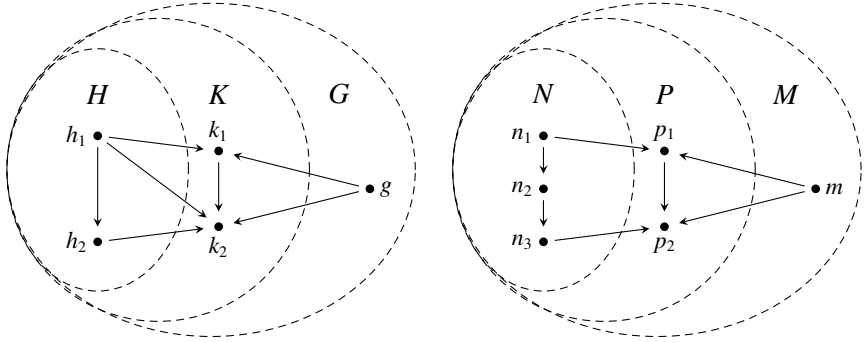
\begin{figure}
\centering
\begin{tikzpicture}

\draw[densely dashed] (1.7,0) circle [x radius=2.7,y radius=2.2];
\draw (3.4,1) node[] {$G$};

\draw[densely dashed] (1,0) circle [x radius=2,y radius=2];
\draw (1.8,1) node[] {$K$};

\draw[densely dashed] (.2,0) circle [x radius=1.2,y radius=1.6];
\draw (.2,1) node[] {$H$};

\begin{scope}[shift={(0,.25)}]
\coordinate (A22) at (.2,0.2); 
\coordinate (A23) at (.2,-1.2); 

\filldraw[fill=black, draw=black] (A22) circle[radius=0.05];
\draw (A22) node[left] {\footnotesize{$h_1$}};
\filldraw[fill=black, draw=black] (A23) circle[radius=0.05];
\draw (A23) node[left] {\footnotesize{$h_2$}};

\coordinate (U22) at (1.8,0); 
\coordinate (U23) at (1.8,-1); 

\filldraw[fill=black, draw=black] (U22) circle[radius=0.05];
\draw (U22) node[above] {\footnotesize{$k_1$}};
\filldraw[fill=black, draw=black] (U23) circle[radius=0.05];
\draw (U23) node[below] {\footnotesize{$k_2$}};

\coordinate (X22) at (3.8,-0.5); 

\filldraw[fill=black, draw=black] (X22) circle[radius=0.05];
\draw (X22) node[right] {\footnotesize{$g$}};

\draw[stealth-, shorten <= 2mm, shorten >= 1.5mm] (A23)--(A22);
\draw[stealth-, shorten <= 2mm, shorten >= 1.5mm] (U23)--(A22);

\draw[stealth-, shorten <= 2mm, shorten >= 1.5mm] (U23)--(U22);

\draw[stealth-, shorten <= 2mm, shorten >= 1.5mm] (U22)--(A22);
\draw[stealth-, shorten <= 2mm, shorten >= 1.5mm] (U23)--(A23);

\draw[stealth-, shorten <= 2mm, shorten >= 1.5mm] (U22)--(X22);
\draw[stealth-, shorten <= 2mm, shorten >= 1.5mm] (U23)--(X22);
\end{scope}

\end{tikzpicture}
\quad
\begin{tikzpicture}

\draw[densely dashed] (1.7,0) circle [x radius=2.7,y radius=2.2];
\draw (3.4,1) node[] {$M$};

\draw[densely dashed] (1,0) circle [x radius=2,y radius=2];
\draw (1.8,1) node[] {$P$};

\draw[densely dashed] (.2,0) circle [x radius=1.2,y radius=1.6];
\draw (.2,1) node[] {$N$};

\begin{scope}[shift={(0,.25)}]
\coordinate (A22) at (.2,0.2); 
\coordinate (A22b) at (.2,-0.5); 
\coordinate (A23) at (.2,-1.2); 

\filldraw[fill=black, draw=black] (A22) circle[radius=0.05];
\draw (A22) node[left] {\footnotesize{$n_1$}};
\filldraw[fill=black, draw=black] (A22b) circle[radius=0.05];
\draw (A22b) node[left] {\footnotesize{$n_2$}};
\filldraw[fill=black, draw=black] (A23) circle[radius=0.05];
\draw (A23) node[left] {\footnotesize{$n_3$}};

\coordinate (U22) at (1.8,0); 
\coordinate (U23) at (1.8,-1); 

\filldraw[fill=black, draw=black] (U22) circle[radius=0.05];
\draw (U22) node[above] {\footnotesize{$p_1$}};
\filldraw[fill=black, draw=black] (U23) circle[radius=0.05];
\draw (U23) node[below] {\footnotesize{$p_2$}};

\coordinate (X22) at (3.8,-0.5); 

\filldraw[fill=black, draw=black] (X22) circle[radius=0.05];
\draw (X22) node[right] {\footnotesize{$m$}};

\draw[stealth-, shorten <= 2mm, shorten >= 1.5mm] (A23)--(A22b);
\draw[stealth-, shorten <= 2mm, shorten >= 1.5mm] (A22b)--(A22);

\draw[stealth-, shorten <= 2mm, shorten >= 1.5mm] (U23)--(U22);

\draw[stealth-, shorten <= 2mm, shorten >= 1.5mm] (U22)--(A22);
\draw[stealth-, shorten <= 2mm, shorten >= 1.5mm] (U23)--(A23);

\draw[stealth-, shorten <= 2mm, shorten >= 1.5mm] (U22)--(X22);
\draw[stealth-, shorten <= 2mm, shorten >= 1.5mm] (U23)--(X22);
\end{scope}

\end{tikzpicture}
\caption{The inclusion \(H \hookrightarrow G\) is in \(\cof_1\) but not in \(\cof_{\mathrm{HR}}\), while \(N \hookrightarrow M\) is in \(\cof_{\mathrm{HR}}\) but not in \(\cof_1\). See \Cref{rem:cofib_not_1cofib}.
}
\label{fig:cofib_not_1cofib}
\end{figure}

The following example demonstrates that \(\cof_1\) and \(\cof_{\mathrm{HR}}\) are not the same. Indeed, we have neither \(\cof_1 \subseteq \cof_{\mathrm{HR}}\) nor \(\cof_{\mathrm{HR}} \subseteq \cof_1\). 

\begin{eg}\label{rem:cofib_not_1cofib}
    Consider the directed graphs in \Cref{fig:cofib_not_1cofib}. The inclusion \(i \colon H \hookrightarrow G\) is in \(\cof_1\): the map \(\pi \colon K \to H\) given by \(\pi(k_1) = h_1\) and \(\pi(k_2) = h_2\) exhibits \(H\) as a strong 1-deformation retract of \(K\). We have \(K=rH\), and one can check that there is no possible choice for \(\rho(k_2)\), so that \(i\) is not in \(\cof_{\mathrm{HR}}\).

    The inclusion \(j \colon N \hookrightarrow M\) is in \(\cof_{\mathrm{HR}}\): we have \(P=rN\) and the assignment \(\rho(p_1)=n_1\) and \(\rho(p_2)=n_3\) fulfills the required condition. But \(j\) is not in \(\cof_1\). To see this, observe that the only non-identity map of graphs from \(P\) to itself that fixes the vertices of \(N\) is the map that sends \(p_1\) to \(n_2\) and \(p_2\) to \(n_3\). This, then, is our only choice for a map \(\pi \colon P \to N\) exhibiting \(N\) as a strong 1-deformation retract. As \(\dis(n_2, p_1) = +\infty = \dis(p_1, n_2)\) we have neither \(j \circ \pi \rightsquigarrow_1 \id_{P}\) nor \(\id_{P} \rightsquigarrow_1 j \circ \pi\), and as there are no other maps of graphs from \(P\) to itself, there is no hope of constructing a longer zig-zag of \(1\)-homotopies.
\end{eg}

Thus, the Brown category structure in \Cref{thm:DiGraph-Brown-cat} is distinct from that in \cite[Theorem 7.2]{HepworthRoff2024}. However, we note that the factorization of the codiagonal is the same in both Brown categories, because the map \(\iota \colon \{-2,2\} \hookrightarrow K_1\) used in the proof of \Cref{prop:cylr-inclusions} lies in \(\cof_1\cap\cof_{\mathrm{HR}}\) and both structures have the same weak equivalences. This implies that to compute homotopy colimits with respect to \(\Ee_1\) in \(\DiGraph\) we can pick any cofibrant replacement of the diagram involved, either using \(\cof_1\), \(\cof_{\mathrm{HR}}\) or \(\cof'_1=\cof_1\cap\cof_{\mathrm{HR}}\). Indeed, it also implies that \((\DiGraph,\Ee_1,\cof'_1)\) is a Brown category.

\begin{rem}
    While we were finishing this paper, the paper~\cite{EIXYZ} by Eldridge \textit{et al} appeared on the arXiv. The authors there define a class of weak equivalences on \(\DiGraph\) called \emph{cubical weak equivalences} and prove that the \(\infty\)-category \(\DiGraph_\infty\) obtained as the \(\infty\)-localization  of \(\DiGraph\) at the class  of cubical weak equivalences is equivalent to the \(\infty\)-category of spaces. We plan to study the \(\infty\)-localizations  \(\DiGraph[\Ee_r^{-1}]\) when \(r\)-varies and compare them with \(\DiGraph_\infty\).
\end{rem}

%-------------------------------------------------------------

\bibliographystyle{plain} 
\bibliography{r-homotopy} 

@article{KapulkinMavinkurve,
title = {The fundamental group in discrete homotopy theory},
journal = {Advances in Applied Mathematics},
volume = {164},
pages = {102838},
year = {2025},
issn = {0196-8858},
doi = {https://doi.org/10.1016/j.aam.2024.102838},
url = {https://www.sciencedirect.com/science/article/pii/S0196885824001702},
author = {Krzysztof Kapulkin and Udit Mavinkurve}
}

@article{Carranza_Kapulkin_2024, 
title={Cubical setting for discrete homotopy theory, revisited}, 
volume={160}, 
DOI={10.1112/S0010437X24007486}, 
number={12}, 
journal={Compositio Mathematica}, 
author={Daneil Carranza and Krzysztof Kapulkin}, 
year={2024}, 
pages={2856–2903}}

@misc{EIXYZ,
    author = {Eldridge, Briony and Ivanov, Sergei O. and Xu, Xiomeng and Yau, Shing-Tung and Zhang, Mengmeng},
    title = {The discrete homotopy hypothesis for directed graphs},
    howpublished = {{\tt arXiv:2605.04959}},
    year = {2026},
}

@article {BBLL,
    AUTHOR = {Babson, Eric and Barcelo, H\'el\`ene and de Longueville, Mark
              and Laubenbacher, Reinhard},
     TITLE = {Homotopy theory of graphs},
   JOURNAL = {J. Algebraic Combin.},
  FJOURNAL = {Journal of Algebraic Combinatorics. An International Journal},
    VOLUME = {24},
      YEAR = {2006},
    NUMBER = {1},
     PAGES = {31--44},
      ISSN = {0925-9899,1572-9192},
   MRCLASS = {05E25 (05C99 18G99 55P99)},
  MRNUMBER = {2245779},
MRREVIEWER = {Alberto\ Cavicchioli},
       DOI = {10.1007/s10801-006-9100-0},
       URL = {https://doi.org/10.1007/s10801-006-9100-0},
}

@article {GLMY2014,
    AUTHOR = {Grigor'yan, Alexander and Lin, Yong and Muranov, Yuri and Yau,
              Shing-Tung},
     TITLE = {Homotopy theory for digraphs},
   JOURNAL = {Pure Appl. Math. Q.},
  FJOURNAL = {Pure and Applied Mathematics Quarterly},
    VOLUME = {10},
      YEAR = {2014},
    NUMBER = {4},
     PAGES = {619--674},
      ISSN = {1558-8599,1558-8602},
   MRCLASS = {55P99 (05C20)},
  MRNUMBER = {3324763},
       DOI = {10.4310/PAMQ.2014.v10.n4.a2},
       URL = {https://doi.org/10.4310/PAMQ.2014.v10.n4.a2},
}

@book {Riehl,
    AUTHOR = {Riehl, Emily},
     TITLE = {Category {T}heory in {C}ontext},
    SERIES = {Aurora Dover Modern Math Originals},
 PUBLISHER = {Dover Publications, Inc., Mineola, NY},
      YEAR = {2016},
     PAGES = {xvii+240},
      ISBN = {978-0-486-80903-8; 0-486-80903-X},
   MRCLASS = {18-01 (18Axx 97H99)},
  MRNUMBER = {4727501},
}

@article{Brown1973,
    author = {Brown, Kenneth S.},
    title = {Abstract homotopy theory and generalized sheaf cohomology},
    journal = {Trans. Amer. Math. Soc.},
    volume = {186},
    year = {1973},
    pages = {419--458},
}

@article{CarranzaKapulkinKim2023,
    author = {Carranza, Daniel and Kapulkin, Krzysztof and Kim, Jinho},
    title = {Nonexistence of Colimits in Naive Discrete Homotopy Theory},
    journal = {Appl.~Categ.~Struct.},
    fjournal = {Applied Categorical Structures},
    volume = {31},
    number = {5},
    year = {2023},
    note = {Article number 41},
}

@book {CartanEilenberg56,
    AUTHOR = {Cartan, Henri and Eilenberg, Samuel},
     TITLE = {Homological {A}lgebra},
 PUBLISHER = {Princeton University Press, Princeton, NJ},
      YEAR = {1956},
     PAGES = {xv+390},
}

@article {celw18,
    AUTHOR = {Cirici, Joana and Egas Santander, Daniela and Livernet, Muriel
              and Whitehouse, Sarah},
     TITLE = {Derived {$A$}-infinity algebras and their homotopies},
   JOURNAL = {Topology Appl.},
  FJOURNAL = {Topology and its Applications},
    VOLUME = {235},
      YEAR = {2018},
     PAGES = {214--268},
}

@article {celw,
    AUTHOR = {Cirici, Joana and Egas Santander, Daniela and Livernet, Muriel
              and Whitehouse, Sarah},
     TITLE = {Model category structures and spectral sequences},
   JOURNAL = {Proc. Roy. Soc. Edinburgh Sect. A},
  FJOURNAL = {Proceedings of the Royal Society of Edinburgh. Section A.
              Mathematics},
    VOLUME = {150},
      YEAR = {2020},
    NUMBER = {6},
     PAGES = {2815--2848},
}

@article {fglw22,
    AUTHOR = {Fu, Xin and Guan, Ai and Livernet, Muriel and Whitehouse,
              Sarah},
     TITLE = {Model category structures on multicomplexes},
   JOURNAL = {Topology Appl.},
  FJOURNAL = {Topology and its Applications},
    VOLUME = {316},
      YEAR = {2022},
     PAGES = {Paper No. 108104, 26},
}

@article{HepworthRoff2023,
    author = {Hepworth, Richard and Roff, Emily},
    title = {The reachability homology of a directed graph},
    journal = {Int. Math. Res. Not. IMRN.},
    fjournal = {International Mathematics Research Notices},
    volume = {2025},
    number = {3},
    pages = {1--18},
    year = {2025},
    note = {rnae280},
}

@misc{HepworthRoff2024,
    author = {Hepworth, Richard and Roff, Emily},
    title = {Bigraded path homology and the magnitude-path spectral sequence},
    howpublished = {{\tt arXiv:2404.06689}},
    year = {2024},
}

@misc {Asao-filtered,
      AUTHOR = {Asao, Yasuhiko},
       TITLE = {Magnitude and magnitude homology of filtered set enriched categories},
HOWPUBLISHED = {{\tt arXiv:2303.05677}}, 
        YEAR = {2023},
}

@article {Asao-path,
    AUTHOR = {Asao, Yasuhiko},
     TITLE = {Magnitude homology and path homology},
   JOURNAL = {Bull.~Lond.~Math.~Soc.},
  FJOURNAL = {Bulletin of the London Mathematical Society},
    VOLUME = {55},
      YEAR = {2023},
    NUMBER = {1},
     PAGES = {375-398},
      ISSN = {0024-6093},
       DOI = {10.1112/blms.12734},
}

@article {CDKOSW,
    AUTHOR = {Carranza, Daniel and Doherty, Brandon and Kapulkin, Krzyzstof
              and Opie, Morgan and Sarazola, Maru and Wong, Liang Ze},
     TITLE = {Cofibration category of digraphs for path homology},
   JOURNAL = {Algebr. Comb.},
  FJOURNAL = {Algebraic Combinatorics},
    VOLUME = {7},
      YEAR = {2024},
    NUMBER = {2},
     PAGES = {475--514},
      ISSN = {2589-5486},
   MRCLASS = {55N35 (55U10 55U40)},
  MRNUMBER = {4741925},
MRREVIEWER = {Daniel\ Graves},
       DOI = {10.5802/alco.341},
       URL = {https://doi.org/10.5802/alco.341},
}

@article{HepworthWillerton2017,
	title = {Categorifying the magnitude of a graph},
	author = {Hepworth, Richard and Willerton, Simon},
	pages = {31--60},
    journal = {Homol.~Homotopy Appl.},	fjournal={Homology, Homotopy and Applications},
	volume={19},
	year={2017},
}

@article {Meier16,
    AUTHOR = {Meier, Lennart},
     TITLE = {Fibration categories are fibrant relative categories},
   JOURNAL = {Algebr. Geom. Topol.},
  FJOURNAL = {Algebraic \& Geometric Topology},
    VOLUME = {16},
      YEAR = {2016},
    NUMBER = {6},
     PAGES = {3271--3300},
    }

@article {DIMZ,
    AUTHOR = {Di, Shaobo and Ivanov, Sergei O. and Mukoseev, Lev and Zhang,
              Mengmeng},
     TITLE = {On the path homology of {C}ayley digraphs and covering
              digraphs},
   JOURNAL = {J. Algebra},
  FJOURNAL = {Journal of Algebra},
    VOLUME = {653},
      YEAR = {2024},
     PAGES = {156--199},
      ISSN = {0021-8693,1090-266X},
   MRCLASS = {05C20 (05C25 20J05 55N35 55T10)},
  MRNUMBER = {4746956},
MRREVIEWER = {Mohammad\ Javad\ Nikmehr},
       DOI = {10.1016/j.jalgebra.2024.05.005},
       URL = {https://doi.org/10.1016/j.jalgebra.2024.05.005},
}

@article {Cisinski2010,
    AUTHOR = {Cisinski, Denis-Charles},
     TITLE = {Cat\'egories d\'erivables},
   JOURNAL = {Bull. Soc. Math. France},
  FJOURNAL = {Bulletin de la Soci\'et\'e{} Math\'ematique de France},
    VOLUME = {138},
      YEAR = {2010},
    NUMBER = {3},
     PAGES = {317--393},
      ISSN = {0037-9484,2102-622X},
   MRCLASS = {18G10 (18G55 55U35 55U40)},
  MRNUMBER = {2729017},
MRREVIEWER = {Timothy\ Porter},
       DOI = {10.24033/bsmf.2592},
       URL = {https://doi.org/10.24033/bsmf.2592},
}

@article {PH-curvature,
    AUTHOR = {Kempton, Mark and M\"{u}nch, Florentin and Yau, Shing-Tung},
     TITLE = {A homology vanishing theorem for graphs with positive curvature},
   JOURNAL = {Comm. Anal. Geom.},
  FJOURNAL = {Communications in Analysis and Geometry},
    VOLUME = {29},
      YEAR = {2021},
    NUMBER = {6},
     PAGES = {1449--1473},
      ISSN = {1019-8385,1944-9992},
   MRCLASS = {53C21 (57M15)},
  MRNUMBER = {4367431},
MRREVIEWER = {Marko\ \v{Z}ivkovi\'{c}},
       DOI = {10.4310/CAG.2021.v29.n6.a5},
       URL = {https://doi.org/10.4310/CAG.2021.v29.n6.a5},
}

@article {LeinsterShulman,
    AUTHOR = {Leinster, Tom and Shulman, Michael},
     TITLE = {Magnitude homology of enriched categories and metric spaces},
   JOURNAL = {Algebr. Geom. Topol.},
  FJOURNAL = {Algebraic \& Geometric Topology},
    VOLUME = {21},
      YEAR = {2021},
    NUMBER = {5},
     PAGES = {2175--2221},
      ISSN = {1472-2747,1472-2739},
   MRCLASS = {18G90 (16E40 51F99 55N31)},
  MRNUMBER = {4334510},
MRREVIEWER = {Ahmet\ A.\ Husainov},
       DOI = {10.2140/agt.2021.21.2175},
       URL = {https://doi.org/10.2140/agt.2021.21.2175},
}

@article {Asao-curvature,
    AUTHOR = {Asao, Yasuhiko},
     TITLE = {Magnitude homology of geodesic metric spaces with an upper
              curvature bound},
   JOURNAL = {Algebr. Geom. Topol.},
  FJOURNAL = {Algebraic \& Geometric Topology},
    VOLUME = {21},
      YEAR = {2021},
    NUMBER = {2},
     PAGES = {647--664},
      ISSN = {1472-2747,1472-2739},
   MRCLASS = {55N35 (51F99)},
  MRNUMBER = {4250513},
MRREVIEWER = {David\ Matthew\ Freeman},
       DOI = {10.2140/agt.2021.21.647},
       URL = {https://doi.org/10.2140/agt.2021.21.647},
}

@article {LeinsterGraph,
    AUTHOR = {Leinster, Tom},
     TITLE = {The magnitude of a graph},
   JOURNAL = {Math. Proc. Cambridge Philos. Soc.},
  FJOURNAL = {Mathematical Proceedings of the Cambridge Philosophical
              Society},
    VOLUME = {166},
      YEAR = {2019},
    NUMBER = {2},
     PAGES = {247--264},
      ISSN = {0305-0041,1469-8064},
   MRCLASS = {05C31 (05C50 18D20 57M15)},
  MRNUMBER = {3903118},
MRREVIEWER = {Lorenzo\ Traldi},
       DOI = {10.1017/S0305004117000810},
       URL = {},
}

@unpublished{GomiGeodesic,
    author = {Gomi, Kiyonori},
    title = {Magnitude homology of geodesic space},
    year = {2019},
    note = { {\tt arXiv:1902.07044}},
}

@article {Deligne,
    AUTHOR = {Deligne, Pierre},
     TITLE = {Th\'{e}orie de {H}odge. {II}},
   JOURNAL = {Inst. Hautes \'{E}tudes Sci. Publ. Math.},
  FJOURNAL = {Institut des Hautes \'{E}tudes Scientifiques. Publications
              Math\'{e}matiques},
    volume = {40},
      YEAR = {1971},
     PAGES = {5--57},
      ISSN = {0073-8301,1618-1913},
   MRCLASS = {14C30 (14F15)},
  MRNUMBER = {498551},
MRREVIEWER = {J.\ H. M. Steenbrink},
       URL = {http://www.numdam.org/item?id=PMIHES_1971__40__5_0},
}

@book {Hatcher,
    AUTHOR = {Hatcher, Allen},
     TITLE = {Algebraic {T}opology},
 PUBLISHER = {Cambridge University Press, Cambridge},
      YEAR = {2002},
     PAGES = {xii+544},
      ISBN = {0-521-79160-X; 0-521-79540-0},
   MRCLASS = {55-01 (55-00)},
  MRNUMBER = {1867354},
MRREVIEWER = {Donald\ W.\ Kahn},
}

@incollection {Gromov,
    AUTHOR = {Gromov, Mikhael},
     TITLE = {Quantitative homotopy theory},
 BOOKTITLE = {Prospects in mathematics ({P}rinceton, {NJ}, 1996)},
     PAGES = {45--49},
 PUBLISHER = {Amer. Math. Soc., Providence, RI},
      YEAR = {1999},
      ISBN = {0-8218-0975-X},
   MRCLASS = {57N65 (53C23 55P99)},
  MRNUMBER = {1660471},
MRREVIEWER = {Vagn\ Lundsgaard\ Hansen},
}

@unpublished {GLMY2013,
    AUTHOR = {Grigor’yan, Alexander and Lin, Yong and Muranov, Yuri and Yau, {S}hing-{T}ung},
    TITLE = {Homologies of path complexes and digraphs},
    YEAR = {2012},
    NOTE = {{\tt arXiv:1207.2834}},
}

@article {PH-eilenbergsteenrod,
    AUTHOR = {Grigor'yan, Alexander and Jimenez, Rolando and Muranov, Yuri
              and Yau, Shing-Tung},
     TITLE = {On the path homology theory of digraphs and
              {E}ilenberg-{S}teenrod axioms},
   JOURNAL = {Homology Homotopy Appl.},
  FJOURNAL = {Homology, Homotopy and Applications},
    VOLUME = {20},
      YEAR = {2018},
    NUMBER = {2},
     PAGES = {179--205},
      ISSN = {1532-0073,1532-0081},
   MRCLASS = {05C20 (05C10 55N40 55P10 57M15)},
  MRNUMBER = {3812462},
       DOI = {10.4310/HHA.2018.v20.n2.a9},
       URL = {https://doi.org/10.4310/HHA.2018.v20.n2.a9},
}

@unpublished {SOIvanov,
AUTHOR = {Ivanov, Sergei O.},
TITLE = {Nested homotopy models of finite metric spaces and their spectral homology},
YEAR = {2023},
NOTE = {{\tt arXiv:2312.11878}},
}

@misc{RB09,
    author = {Radulescu-Banu, Andrei},
    title = {Cofibrations in homotopy theory},
    howpublished = {{\tt arXiv:math/0610009v4}},
    year = {2009},
}

\end{document}